\documentclass[12pt,draftcls,onecolumn]{IEEEtran}

\usepackage{hyperref}  

\usepackage{xcolor}
\usepackage{pdfsync}
\usepackage{graphicx}
\graphicspath{{figures/}}
\usepackage{epsfig}
\usepackage{amsmath}
\usepackage{amssymb}
\usepackage{amsthm}
\usepackage{mathtools}
\usepackage{algorithm}
\usepackage{algorithmicx}
\usepackage{algpseudocode}
\usepackage{amsfonts}
\usepackage{lineno}
\usepackage{tikz}
\usetikzlibrary{positioning,shapes}
\usetikzlibrary{arrows}

\allowdisplaybreaks
\newtheorem{theorem}{Theorem}
\newtheorem{definition}{Definition}
\newtheorem{corollary}{Corollary}
\newtheorem{lemma}{Lemma}
\newtheorem{assumption}{Assumption}
\newtheorem{proposition}{Proposition}
\newtheorem{remark}{Remark}

\newtheorem{problem}{Problem}

\def\field#1{\mathbb #1}%
\def\R{\field{R}}%
\def\N{\field{N}}%
\def\Z{\field{Z}}%
\def\B{\field{B}}%
\newcommand{\X}{\ensuremath{\mathcal X}}

\newcommand{\U}{\ensuremath{\mathcal U}}

\newcommand{\Rp}{\R_{\geq 0}}

\newcommand{\Zp}{\Z_{\geq 0}}
\def\K{\mathcal{K}}%
\DeclareMathOperator{\id}{id}

\def\KL{\mathcal{KL}}%
\def\Kinf{\mathcal{K}_\infty}%

\def\A{\mathcal{A}}

\def\V{\mathcal{V}}

\def\Ni{\mathcal{N}_i}

\def\C{\mathcal{C}}
\def\Ci{\mathcal{C}_i}
\def\Cij{\mathcal{C}_i (j)}

\def\dist{\mathrm{dist}}
\def\diam{\mathrm{diam}}

\def\ri{\mathrm{ri}}
\def\co{\mathrm{co}}
\def\rank{\mathrm{rank}}

\def\Zb{\bar{\mathcal{Z}}}
\def\zbi{\bar{z}_i}
\let\rb=\partial
\usepackage{color}

\usepackage{xspace}

\providecommand{\afcl}{\textcolor{black}{admissible finite-step feedback control law}\xspace}
\providecommand{\afcls}{\textcolor{black}{admissible finite-step feedback control laws}\xspace}

\usepackage{enumerate}

\title{Distributed Multi-Step Model Predictive Control for Consensus}

\author{Navid~Noroozi%
\thanks{N. Noroozi is with SIGNON Deutschland GmbH, R\"udesheimer Str.~7,
80686 Munich, Germany (e-mail: \href{mailto:navid.n.noroozi@deutschebahn.com}{navid.n.noroozi@deutschebahn.com}).}%
}

\begin{document}
\maketitle
\begin{abstract}
This paper studies consensus of discrete-time multi-agent systems under time-varying directed communication, state and input constraints using a distributed multi-step model predictive control (MPC) framework. 
Consensus is recast as stabilization of the agreement set, and a geometric viewpoint based on convex-hull invariance and strict interiority is adopted. Building on an existing geometric necessary and sufficient condition for agreement, we show that enforcing terminal inclusion in local neighbor convex hulls guarantees hull invariance but does not, in general, imply the strict relative-interior property required for convergence. An explicit counterexample demonstrates that strictness cannot be deduced from feasibility and contraction constraints alone.

To resolve this issue without shrinking feasible sets or altering primary performance objectives, a lexicographic tie-breaking mechanism is introduced. Among optimal (or near-optimal) MPC solutions, the proposed secondary criterion selects trajectories maximizing an interiority measure with respect to the neighbor hull. It is shown that whenever an interior feasible terminal state exists, this selection rule enforces the strictness condition required for asymptotic consensus. Explicit horizon conditions are derived for single- and double-integrator agents with bounded inputs, ensuring feasibility and automatic existence of interior feasible terminal points. The resulting scheme provides a distributed and implementable route to consensus via finite-step set-Lyapunov contraction. Numerical simulations with distributed inter-process communication illustrate monotone diameter decay and report per-agent computational complexity.
\end{abstract}

\tableofcontents

\section{Introduction}

Consensus and agreement in networked multi-agent systems are fundamental in distributed estimation, cooperative robotics, and large-scale cyber--physical systems.
A classical geometric viewpoint interprets consensus as \emph{set stabilization} of the agreement set and characterizes convergence through convex-hull invariance and strictness properties of the induced state update map.
In particular, Moreau's seminal results~\cite{Moreau.2005} establish necessary and sufficient conditions for asymptotic agreement under time-varying directed communication, relying on (i) convex-hull inclusion and (ii) a strictness condition excluding boundary-invariant trajectories, together with suitable connectivity assumptions.

Model predictive control (MPC) provides a systematic framework to incorporate constraints, performance objectives, and prediction models.
This has motivated a large body of work on distributed MPC (DMPC) for cooperative control and agreement.
Early distributed receding-horizon schemes addressed stabilization and formation problems by coupling local MPC updates through exchanged predicted trajectories and by imposing compatibility constraints to ensure closed-loop stability~\cite{Dunbar.2006,Gruene.2012a}.
Building on this direction, Ferrari-Trecate \emph{et al.} developed MPC schemes tailored to consensus of agents with single- and double-integrator dynamics under bounded inputs and time-varying communication, with convergence proofs that leverage geometric properties of optimal trajectories and Moreau-type strictness arguments \cite{Ferrari-Trecate.2009}.
Their work remains one of the closest references to the present paper in terms of problem formulation (consensus under constraints and directed switching graphs) and emphasis on implementable MPC protocols for integrator models.

Beyond integrator consensus, DMPC consensus has been studied for broader linear and nonlinear multi-agent models and for additional network effects.
For example, robust DMPC frameworks have been proposed to handle input constraints together with bounded time-varying communication delays, establishing recursive feasibility and robust convergence for linear discrete-time MASs~\cite{Wei.2015}.
Consensus-oriented DMPC has also been combined with reduced communication mechanisms such as (self-)triggered or intermittent exchange of predicted trajectories while maintaining feasibility and convergence guarantees \cite{Zhan.2018}.
For nonlinear or higher-order agents, distributed MPC designs have been investigated to achieve consensus-like objectives under constraints (e.g., second-order nonlinear MASs) using tailored terminal ingredients and feasibility arguments \cite{Gao.2017}.
More recently, the scope has expanded toward \emph{optimal} and \emph{output} consensus of heterogeneous agents over directed graphs, where the DMPC problem is formulated to balance transient performance with convergence to an optimal agreement point \cite{Bai.2025}.
Complementary tutorial and survey articles summarize algorithmic paradigms and stability mechanisms in DMPC, including coupling architectures, feasibility preservation, and convergence tools~\cite{Scattolini.2009}.

\subsection{Motivation and gap}
Despite this progress, a recurring conceptual difficulty in DMPC consensus is the role of \emph{strictness}.
Moreau-type consensus theorems require, in one form or another, that when disagreement is present, the next iteration enters the relative interior of a suitable neighbor hull, thereby preventing trajectories from remaining trapped on invariant faces of the global convex hull \cite{Moreau.2005}.
In many MPC-based protocols, convex-hull \emph{invariance} can be enforced directly by terminal constraints, but \emph{strict} interiority is often obtained indirectly via problem-specific geometric arguments (e.g., nonzero first step for integrators) or by strengthening terminal conditions in ways that may shrink feasibility sets and complicate robustness claims \cite{Ferrari-Trecate.2009}.
For general nonlinear dynamics and general convex constraints, it is therefore natural to ask:
\emph{Which strictness property is truly necessary for consensus, and how can it be enforced by an MPC implementation without altering feasibility?}

\subsection{Contributions}
This paper proposes a distributed multi-step MPC framework for asymptotic consensus in discrete-time networks that makes the strictness issue explicit and resolves it with a minimal implementation mechanism.
Our approach combines (i) a design-neutral geometric consensus core, (ii) a finite-step Lyapunov viewpoint for the agreement set, and (iii) a feasibility-preserving tie-breaking rule that enforces the required interiority whenever it is achievable.

\emph{Consensus as finite-step set stabilization:}
We recast consensus as stabilization of the agreement set and employ the disagreement diameter as a proper set-Lyapunov candidate.
We restate a Moreau-type necessary/sufficient consensus theorem in our notation for time-varying directed graphs, highlighting that convex-hull inclusion and a strictness (relative-interior) property form the essential geometric ingredients.

\emph{A distributed multi-step MPC template enforcing hull inclusion and finite-step barycenter contraction:}
We introduce a local MPC problem that (a) constrains the terminal state to lie in the local neighbor convex hull, guaranteeing convex-hull invariance, and (b) enforces a finite-step contraction inequality toward a local barycenter.
This yields a modular route to consensus via finite-step Lyapunov decay~\cite{Geiselhart.2014c}, provided that strictness can be ensured along closed-loop trajectories.
We note that the \emph{feasibility of a finite-step Lyapunov decay is ensured by a constructive converse control finite-step Lyapunov function theorem}, see~\cite{Noroozi.2020} for more details.

\emph{Impossibility of deducing strictness from the baseline OCP and a minimal remedy:}
We show by an explicit counterexample that, even under feasibility and joint connectivity, the baseline OCP does \emph{not} imply the required relative-interior strictness: feasible (and even cost-improving) solutions may lie on the boundary of the neighbor hull, allowing boundary-invariant behavior.
To remedy this without shrinking the feasible set or changing the primary objective value, we introduce a two-stage \emph{lexicographic} selection rule:
among all (near-)optimal primary solutions, select one that maximizes an interiority measure (distance to the hull boundary).
We prove that whenever an interior feasible terminal point exists, this tie-break enforces relative-interior updates and restores the strictness condition required by the geometric consensus theorem.

\emph{Explicit integrator specializations and computational evidence:}
For single- and double-integrator agents with input bounds (and no additional state constraints), we derive explicit horizon conditions guaranteeing feasibility and automatic existence of interior feasible terminal points, leading to consensus under the proposed lexicographic implementation.
We provide distributed numerical simulations (including per-agent solve-time statistics and problem dimensions) demonstrating monotone decay of the diameter and illustrating if the lexicographic stage is activated or not in practice.

\subsection{Relation to prior DMPC consensus work}
The closest comparison is to \cite{Ferrari-Trecate.2009}, which studies MPC consensus for single- and double-integrator models and proves strictness using integrator-specific optimal-trajectory geometry.
In contrast, the present paper (a) separates feasibility from strictness at the level of implemented closed-loop updates, (b) provides an impossibility result explaining why strictness cannot be expected from a generic terminal hull constraint plus a contractive inequality alone, and (c) enforces strictness via a minimal feasibility-preserving selection rule.
Compared with robust/delay-aware DMPC consensus designs such as \cite{Wei.2015}, our emphasis is not on delay compensation but on identifying and enforcing the geometric strictness mechanism needed for agreement under general nonlinear dynamics and time-varying directed graphs.
Finally, relative to recent optimal/output-consensus DMPC formulations over directed graphs \cite{Bai.2025}, our focus is on a set-stabilization/finite-step Lyapunov pathway to \emph{asymptotic} agreement, together with an implementation mechanism that guarantees the strictness condition required by the underlying consensus theorem.

under feasibility and connectivity.

\subsection{Organization}
Section~\ref{sec:preliminaries} introduces the preliminary mathematical notation.
Section~\ref{sec:problem-formulation} formulates the problem setting and the multi-step information pattern.
Section~\ref{sec:nec-suf-moreau} establishes the geometric consensus conditions and their connection to set stabilization.
Section~\ref{subsec:dmpc} presents the distributed multi-step MPC formulation for the consensus problem of interest.
Sections~\ref{sec:impossibility} and~\ref{sec:OCP_with_lex_sel}, respectively, analyze why the baseline OCP does not ensure strictness and develop the lexicographic selection mechanism, culminating in the main consensus theorem.
Section~\ref{sec:linear-dynamics} specifies the results for agents with linear dynamics. In particular, it provides explicit integrator specializations.
Section~\ref{sec:numerical-simulations} reports distributed numerical simulations and computational complexity measures.

\section{Preliminaries}
\label{sec:preliminaries}
This section covers notational and conceptual preliminaries, required in the next section.

\paragraph{Norms and balls}
Let $\|\cdot\|$ be a norm on $\mathbb{R}^d$. 
For $r\ge 0$, $c \geq 0$, denote $\mathbb{B}_r(c):=\{x\in\mathbb{R}^d:\|x - c \|\le r\}$.
For $c = 0$, we drop the argument $c$ from the notation, i.e. $\mathbb{B}_r := \mathbb{B}_r(0)$, moreover, $\mathbb{B}:=\mathbb{B}_1$. 
On $(\mathbb{R}^d)^\ell$ we use the product Euclidean norm
\[
\|x\| := \Big(\sum_{i=1}^{\ell}\|x_i\|^2\Big)^{1/2}.
\]

\paragraph{Distance to a closed set}
For a nonempty closed set $\A \subseteq \mathbb{R}^n$, define the distance
\[
|x|_A := \inf_{y\in \A}\|x-y\|.
\]
For $\A$ closed, $|x|_\A=0\iff x\in \A$, and $x\mapsto |x|_\A$ is Lipschitz.

\paragraph{Convex hull}
For a finite set $\{x_1,\dots,x_m\}$ define its convex hull
\[
\operatorname{co}\{x_1,\dots,x_m\}:=\Big\{\sum_{k=1}^m\lambda_k x_k:\ \lambda_k\ge 0,\ \sum_{k=1}^m\lambda_k=1\Big\}.
\]
In the sequel, we often abbreviate $\C := \operatorname{co}\{x_1,\dots,x_m\}$.
Let $\mathrm{aff}(\C)$ denote the affine hull of $\C$. We interpret $\partial \C$ and $\mathrm{dist}(\cdot,\partial \C)$
as \emph{relative} notions within $\mathrm{aff}(\C)$, i.e.\ $\partial \C$ is the topological boundary of $\C$ as a subset of $\mathrm{aff}(\C)$, and distances are Euclidean distances in $\mathbb R^d$
(which agree with distances measured within $\mathrm{aff}(\C)$).

\paragraph{Comparison function classes}
A continuous function $\alpha:\mathbb{R}_{\ge 0}\to\mathbb{R}_{\ge 0}$ is in class $\mathcal{K}$
if $\alpha$ is strictly increasing and $\alpha(0)=0$. It is in $\mathcal{K}_\infty$ if additionally $\alpha(s)\to\infty$ as $s\to\infty$.
A function $\beta:\mathbb{R}_{\ge 0}\times\mathbb{R}_{\ge 0}\to\mathbb{R}_{\ge 0}$ is in $\mathcal{KL}$ if
$\beta(\cdot,t)\in\mathcal{K}$ for all $t$ and $\beta(s,t)\to 0$ as $t\to\infty$ for each fixed $s$.
The identity function is denoted by $\id$.

\section{Problem Formulation}\label{sec:problem-formulation}

Consider $\ell\in\mathbb{N}$ agents with discrete-time dynamics
\begin{equation}\label{eq:agent_dyn}
x_i(t+1)=g_i(x_i(t),u_i(t)),\qquad i=1,\dots,\ell,
\end{equation}
where $x_i(t)\in \X_i\subseteq\mathbb{R}^d$ and $u_i(t)\in \U_i\subseteq\mathbb{R}^{m_i}$.

\begin{assumption}[$\K$-boundedness]\label{ass:K-boundedness}
For each $i$, the map $g_i:\mathbb{R}^d\times\mathbb{R}^{m_i}\to\mathbb{R}^d$ is continuous and locally Lipschitz in $(x_i,u_i)$.
Moreover, it is $\K$-bounded: there exist $\kappa_{1,i},\kappa_{2,i}\in\K_\infty$ such that
\[
\|g_i(x_i,u_i)\|\le \kappa_{1,i}(\|x_i\|)+\kappa_{2,i}(\|u_i\|),\qquad \forall (x_i,u_i) \in ( \X_i, \U_i ).
\]
\end{assumption}

Assumption~\ref{ass:K-boundedness} ensures well-posedness of trajectories under bounded inputs and is standard in nonlinear MPC and nonlinear systems stability analysis~\cite{Grune.2014,Geiselhart.2014c}.

\subsection{Graph theory and connectivity}
\label{subsec:graph_stability}

\paragraph{Communication graph and neighbors.}
At outer time index $j$ (defined below), the communication structure is a directed graph
$\mathcal{G}(j)=(\mathcal{V},\mathcal{E}(j))$, with $\mathcal{V}=\{1,\dots,\ell\}$.
The neighbor set of agent $i$ at time $j$ is
\[
\Ni(j):=\{k\in\mathcal{V}:\ (k,i)\in\mathcal{E}(j)\}.
\]

\paragraph{Joint connectivity (bounded window).}
\begin{assumption}[Joint spanning-tree connectivity]\label{ass:joint_tree}
There exists $B\in\mathbb{N}$ such that for every $j \in \Zp$, the union graph
\[
\mathcal{G}[j,j+B-1] := (\mathcal{V},\ \cup_{s=j}^{j+B-1}\mathcal{E}(s))
\]
contains a directed spanning tree. In the bidirectional case, this is equivalent to joint connectivity.
\end{assumption}

\subsection{Finite-step control Lyapunov function and barycenter tracking}
\paragraph{Outer sampling and information pattern}
Fix $M\in\N$. 
Define sampled states (outer index)
\begin{equation}
z_i(j):=x_i(jM),\qquad z(j)=(z_1(j),\dots,z_\ell(j)).
\end{equation}
At each outer time $j$, each agent receives the neighbor samples $\{z_k(j):k\in\mathcal{N}_i(j)\}$ \emph{only} and then
computes an $M$-step open-loop control sequence to apply on inner times $t=jM,\dots,jM+M-1$.

\paragraph{Local barycenter.}
Define the local (time-varying) barycenter as
\begin{equation}\label{eq:barycenter}
\bar z_i(j):=\frac{1}{1+|\mathcal{N}_i(j)|}\sum_{k\in\{i\}\cup\mathcal{N}_i(j)} z_k(j).    
\end{equation}
We denote the set covering the evolution of all possible local barycenter trajectories $\zbi$ over time by $\Zb_i$.

\paragraph{Finite-step admissible control}
Fix $M\in\N$.
Given an outer time $j \geq 0$, let ${\bf u}_i = (u_i(0), u_i(1), \ldots)$ denote a possibly infinite control sequence for agent~\eqref{eq:agent_dyn}, where $u_i(t)\in \U_i$ for all $t=0,1,\ldots$.
If we only study trajectories of~\eqref{eq:agent_dyn} over a finite horizon, we might restrict to finite control sequences denoted by $\mathbf{u}_i(0:k-1) := (u_i(0),\dots,u_i(k-1)) \in \U_i^k$.

Denote $e_i (jM) := x_i(jM) - \bar{z}(j)$, with $x_i(jM) \in \X_i$ and $\bar{z}_i(j)$ as in~\eqref{eq:barycenter}, $\bar{z}_i \in \Zb_i$.
We aim to compute any control sequence $\bar{u}_i$ such that the state $x_i$ tracks the local barycenter trajectory $\bar{z}_i$, i.e. the error trajectory $\|e_i(jM)\| \to 0$ as $j \to \infty$.

\begin{definition} \label{def:afcl}
Let $M \in \N$ and $j \in \Zp$ be given.
Consider an agent~\eqref{eq:agent_dyn} satisfying Assumption~\ref{ass:K-boundedness} and a map $q_i: \X_i \times \Zb_i \to \U_i^M$.
Partition $q_i =: (q_i^1,\dots,q_i^M)$, where $q_i^k \in \U_i$ for each $k = 1,\dots,M$.
The map ${q_i}$ is called an \emph{\afcl} if for all $x_i(jM) \in \X_i, \bar{z}_i (j) \in \Zb_i$ and all $k = 1,\dots,M$, $g_i \big(x_i (jM+k-1),q_i^k (x_i(jM),\bar{z}_i(j) )\big)\in \X_i$.
\end{definition}
Now we define a notion of asymptotic barycenter tracking for each agent $i$.
\begin{definition}[Asymptotic barycenter tracking] \label{def:barycenter-tracking}
Consider a network of agents~\eqref{eq:agent_dyn} with an associated directed graph $(\V,\mathcal{E}(j))$.
Let $M \in \N$ be given.
For each agent $i \in \V$, $q_i \colon \X_i \times \Zb_i \to \U_i^M$ is an \afcl.
We say that agent $i$ asymptotically tracks its associated local barycenter trajectory $\zbi$ under the local \afcl $q_i$ if there exists $\beta_i \in \KL$ such that for all $x_i(0) \in \X_i$, all $j \in \Zp$ and all $\zbi (j) \in \Zb_i$ we have
\begin{equation} \label{eq:asymp-estimate}
\|x_i (jM) - \zbi (j)\| \leq \beta_i\big(\|x_i(0) - \bar{z}_i(0)\|, j \big) .
\end{equation}
\end{definition}
Following the converse Lyapunov function~\cite[Theorem~10]{Noroozi.2020}, under the decay condition $\beta(r,M) < r$ for $\beta\in\KL$ in~\eqref{eq:asymp-estimate},
the existence of \afcl implies any scaling version of the norm function $\|\cdot\|$ itself is a so-called \emph{finite-step} control Lyapunov
function for the system
converse control finite-step Lyapunov function theorem, see~\cite{Noroozi.2020} for more details.
This can be technically summarized as follows, which is a restatement of~\cite[Theorem~10]{Noroozi.2020} in our current context:

\begin{proposition}\label{prop:fsCLF}
Let $M \in \N$ and $j \in \Zp$ be given.
Consider an agent~\eqref{eq:agent_dyn} and an \afcl $q_i: \X_i \times \Zb_i \to \U_i^M$.
Assume that the resulting closed-loop system satisfies~\eqref{eq:asymp-estimate} with $\beta(r,M) < r$ for all $r > 0$.
Then we have that for all $j \geq 0$
\begin{equation} \label{eq:M-decay}
\|x_i ((j+1)M) - \bar{z}_i (j)\| \leq \alpha\big(\|x_i(jM) - \bar{z}_i(j)\|\big) ,
\end{equation}
where $\alpha (\cdot):= \beta(\cdot,M)$.
\end{proposition}
The decay inequality~\eqref{eq:M-decay} will be encoded later into our optimization control problem to enforce the state variables $x_i$ towards the local barycenter $\zbi$, see~\eqref{eq:OCP} below.
Moreover, we note that the $\K$-boundedness condition as in Assumption~\ref{ass:K-boundedness} is necessary to conclude Proposition~\ref{prop:fsCLF}.
\subsection{Consensus problem and its reformulation as set stabilization}
\label{subsec:consensus_problem}

We aim to relate a consensus problem to a set stabilization problem.
Define the consensus set in $(\mathbb{R}^d)^\ell$ as
\[
\A:=\{x=(x_1,\dots,x_\ell): x_1=\cdots=x_\ell\}.
\]
Consensus is indeed set stabilization with respect to $\A$.
Towards this end, we introduce the following notions.
\paragraph{Local neighbor hull}
Given a network of agents~\eqref{eq:agent_dyn} with an associated directed graph
$\mathcal{G}(j)=(\mathcal{V},\mathcal{E}(j))$, define the \emph{local neighbor hull} as
\begin{equation}\label{eq:Ci_def_moreau}
\mathcal{C}_i(j):=\operatorname{co}\Big(\{z_i(j)\}\cup\{z_k(j):k\in\mathcal{N}_i(j)\}\Big) ,
\end{equation}

and the \emph{global hull} is defined by
\begin{equation}\label{eq:H_def_moreau}
H(z(j)):=\operatorname{co}\{z_1(j),\dots,z_\ell(j)\}.
\end{equation}
Now define the \emph{disagreement diameter} $V : \left(\R^d\right)^\ell \to \Rp$ by
\begin{equation}\label{eq:V_diameter}
V(z(j)):=\operatorname{diam}(H(z(j)))=\max_{p,q\in\mathcal{V}}\|z_p(j)-z_q(j)\|.
\end{equation}

\begin{lemma}[Diameter is a proper set-Lyapunov candidate]\label{lem:V_props}
Let $\A=\{z_1=\cdots=z_\ell\}$. Then:
\begin{enumerate}
\item $V(z)=0\iff z\in \A$ and $V(z)>0$ otherwise.
\item $V$ is continuous (globally Lipschitz) in $z$.
\item With $|z|_\A$ defined w.r.t.\ the product Euclidean norm, one has
\begin{equation}\label{eq:V_bounds}
\frac{1}{\sqrt{\ell}}|z|_\A \le V(z)\le 2|z|_\A.
\end{equation}
\end{enumerate}
\end{lemma}

\begin{proof}
(1) is immediate from \eqref{eq:V_diameter}. (2) holds because $V$ is the maximum of finitely many continuous functions.

For (3), let $\bar z:=\frac{1}{\ell}\sum_{i=1}^\ell z_i$. Then $|z|_\A=(\sum_i\|z_i-\bar z\|^2)^{1/2}$.
For any $p,q$,
\[
\|z_p-z_q\|\le \|z_p-\bar z\|+\|z_q-\bar z\|\le 2|z|_\A,
\]
hence $V(z)\le 2|z|_\A$.
Moreover, for any $i$,
\[
\|z_i-\bar z\|=\Big\|\frac{1}{\ell}\sum_{k=1}^\ell (z_i-z_k)\Big\|
\le \frac{1}{\ell}\sum_{k=1}^\ell \|z_i-z_k\|\le V(z).
\]
Thus $|z|_\A \le \sqrt{\ell}\max_i\|z_i-\bar z\|\le \sqrt{\ell}V(z)$, i.e.\ $V(z)\ge \frac{1}{\sqrt{\ell}}|z|_\A$.
\end{proof}
\begin{definition}[Asymptotic consensus]\label{def:asymp_consensus}
Consider the diameter function $V$ as in~\eqref{eq:V_diameter} for a network of agents~\eqref{eq:agent_dyn} with an associated directed graph
$\mathcal{G}(j)=(\mathcal{V},\mathcal{E}(j))$.
The network achieves asymptotic consensus if 
\[
\limsup_{j\to\infty} V\big(z(j)\big) \to 0 .
\]
\end{definition}

\section{Necessary and Sufficient Geometric Conditions for Consensus}\label{sec:nec-suf-moreau}
In~\cite{Moreau.2005}, Moreau presents a number of necessary and sufficient conditions for consensus for agents with nonlinear discrete-time dynamics.
Towards this end, he provides a (strict) convexity condition, given by~\cite[Assumption~1]{Moreau.2005} plus a number of graph connectivity conditions, see~\cite[Theorems~2,~3]{Moreau.2005}.
In this section, we reformulate these ~\cite[Assumption~1]{Moreau.2005} to our wording context.
\begin{assumption}[strict convexity]
\label{ass:asymp_moreau_A1}
Let $z(j)=(z_1(j),\dots,z_\ell(j))\in(\mathbb R^d)^\ell$ denote the agent variables at outer step $j\in\mathbb Z_{\ge0}$.
For each $i$ define the local hull
$
\Cij=\operatorname{co}\big(\{z_i(j)\}\cup\{z_k(j):k\in\mathcal N_i(j)\}\big).
$
There exists a constant $\eta\in[0,1)$ such that for every $i$ and $j$:
\begin{enumerate}
\item \textbf{convex-hull inclusion:}
\begin{equation}
z_i(j+1)\in \Cij .
\label{eq:A1a_eps_hull}
\end{equation}
\item \textbf{Relative interior (RI) strictness when disagreement is nonzero:}
If $z_i(j)$ is not equal to all neighbor values, i.e. if
\[
z_i(j)\notin \bigcap_{k\in\mathcal N_i(j)}\{z_k(j)\},
\]
then
\begin{equation}
z_i(j+1)\in \ri\big(\Cij\big).
\label{eq:A1b_ri_strictness}
\end{equation}
where $\mathrm{diam}(\mathcal C):=\sup\{\|x-y\|:x,y\in\mathcal C\}$ and $\partial \mathcal C$ is the boundary.
\end{enumerate}
\end{assumption}

Now we pack and restate~\cite[Theorem~1, Theorem~2]{Moreau.2005} in the following theorem.

\begin{theorem}[necessary and sufficient conditions for asymptotic consensus]
\label{thm:moreau_nes_suf}
Consider a network of agents~\eqref{eq:agent_dyn} with an associated directed graph $(\mathcal{V},\mathcal{E}(j))$.
Consider any trajectory $z(\cdot)$ generated by a distributed update law $u(\cdot)$.

\smallskip
\noindent\textbf{(Necessity)}  
Assume that asymptotic consensus holds for all initial conditions and for all trajectories of the update law,
and assume further that the update law is \emph{local} in the sense that $z_i(j+1)$ depends only on
$\{z_k(j):k\in\{i\}\cup\mathcal N_i(j)\}$.
Then necessarily:
\begin{enumerate}
\item[\textbf{(N1)}] the hull inclusion \eqref{eq:A1a_eps_hull} holds;
\item[\textbf{(N2)}] the joint spanning-tree condition of Assumption~\ref{ass:joint_tree} holds.
\end{enumerate}
\smallskip
\noindent\textbf{(Sufficiency)}  
If Assumptions~\ref{ass:asymp_moreau_A1} and \ref{ass:joint_tree} hold, then $z(\cdot)$ satisfies
asymptotic consensus in the sense of Definition~\ref{def:asymp_consensus}.
\end{theorem}

\begin{remark}[Counterparts in \cite{Moreau.2005}]\label{rem:moreau_counterparts}
Assumption~\ref{ass:asymp_moreau_A1} is the point-valued specialization of Moreau's convexity and strictness
assumptions (see \cite[Assumption~1]{Moreau.2005}), where our hull inclusion \eqref{eq:A1a_eps_hull} corresponds
to his convexity property and our RI strictness \eqref{eq:A1b_ri_strictness} corresponds to his
strict inclusion property (expressed as RI membership when disagreement is nonzero).
The \textbf{sufficiency} part of Theorem~\ref{thm:moreau_nes_suf} is the directed-graph asymptotic-consensus result
in \cite[Theorem~2]{Moreau.2005}, restated in our notation. 
The \textbf{necessity} part stated here is a standard locality-based necessity:
without (N1) the global hull may expand, and without (N2) the network can be split into groups without mutual
influence, so agreement for all initial conditions is impossible.
\end{remark}

\begin{proof}[Proof of Theorem~\ref{thm:moreau_nes_suf}]
We prove necessity and sufficiency separately.

\paragraph{Necessity of (N1)}
Assume asymptotic consensus holds for all initial conditions and the update law is local.
Suppose by contradiction that \eqref{eq:A1a_eps_hull} fails for some $i$ and $j$, i.e.
$z_i(j+1)\notin \Cij$ for some admissible configuration $z(j)$.
Since $\Cij$ is convex and closed, by the separating-hyperplane theorem there exists a vector
$v\in\mathbb S^{d-1}$ such that
\[
\langle v, z_i(j+1)\rangle \;>\; \sup_{x\in \Cij} \langle v,x\rangle.
\]
Because $\Cij\subseteq H(z(j)):=\co\{z_1(j),\dots,z_\ell(j)\}$, the right-hand side is at most
$\sup_{x\in H(z(j))}\langle v,x\rangle = \max_k \langle v,z_k(j)\rangle$. Hence
\[
\max_k \langle v,z_k(j+1)\rangle \;\ge\; \langle v,z_i(j+1)\rangle
\;>\; \max_k \langle v,z_k(j)\rangle.
\]
Thus the support function of the global hull strictly increases in direction $v$, meaning the global hull
expands in that direction. By locality one can fix all agents outside $\{i\}\cup\mathcal N_i(j)$ so that this
expansion cannot be compensated by other agents at time $j+1$. This contradicts asymptotic consensus for all
initial conditions. Therefore \eqref{eq:A1a_eps_hull} is necessary.

\paragraph{Necessity of (N2)}
Assume by contradiction that Assumption~\ref{ass:joint_tree} fails. Then for every $B\in\mathbb N$ there exists
a time $j_B$ such that the union graph $\mathcal G[j_B,j_B+B-1]$ has no directed spanning tree. In particular,
there exists a nontrivial partition of the node set $V=S\cup T$ with $S,T\neq\emptyset$ such that no node in $T$
is reachable from $S$ over this union window. Choose initial conditions at time $j_B$ as two distinct constants:
$z_i(j_B)=a$ for $i\in S$, $z_i(j_B)=b$ for $i\in T$ with $a\neq b$. By locality, nodes in $T$ depend only on
states within their in-neighborhoods, and because $S$ cannot reach $T$ over the union window, the evolution of
$z_T$ over $[j_B,j_B+B]$ is independent of $z_S$. Consequently, one can select (or equivalently, construct) a
trajectory consistent with locality in which the two groups remain separated, contradicting consensus for all
initial conditions. Thus joint spanning-tree connectivity is necessary.

\paragraph{Sufficiency}
Assume Assumptions~\ref{ass:asymp_moreau_A1} and \ref{ass:joint_tree}.
Recal the global hull $H(z(j))$ as in~\eqref{eq:H_def_moreau} and its diameter $V(z(j))$ as in~\eqref{eq:V_diameter}.

\emph{Step 1 (nested hulls).} From \eqref{eq:A1a_eps_hull}, for every $i$ we have
$z_i(j+1)\in \Cij\subseteq H(z(j))$, hence $H(z(j+1))\subseteq H(z(j+1))$ and therefore $V(z(j+1))\le V(z(j))$.
Thus $V(z(j))$ is nonincreasing and converges to some limit $\limsup_{j \to \infty} \geq 0$.

\emph{Step 2 (leaving vertices under strictness).} Fix $j$ with $V(z(j))>0$.
Let $p$ be a vertex (extreme point) of $H(z(j))$ and let $I_p(j):=\{i: z_i(j)=p\}$ be the set of agents at that
vertex. Consider any $i\in I_p(j)$.
If all neighbor values equal $p$, then $\Cij=\{p\}$ and $z_i(j+1)=p$.
Otherwise $\Cij$ contains at least one point different from $p$, hence $p\in \partial \Cij$
and therefore $p\notin \ri(\Cij)$. By \eqref{eq:A1b_ri_strictness} we obtain
$z_i(j+1)\in \ri(\Cij)$, in particular $z_i(j+1)\neq p$. Thus an agent can stay at a vertex $p$ only
if all its current neighbor values equal $p$.

\emph{Step 3 (strict hull contraction over a connectivity window).}
Let $B$ be the window length from Assumption~\ref{ass:joint_tree}. Consider the union graph
$\mathcal G[j,j+B-1]$ which contains a directed spanning tree. Since $V(z(j))>0$, not all agents lie at the same
point. Hence there exists at least one vertex $p$ of $H(z(j))$ such that $I_p(j)\neq V$.
By the spanning-tree property, information from outside $I_p(\cdot)$ must enter $I_p(\cdot)$ somewhere over the
window; otherwise $I_p(\cdot)$ would be closed under in-neighbors across the union graph, preventing reachability
of all nodes. Hence there exists a time $\tau\in\{j,\dots,j+B-1\}$ and an edge $(k,i)\in \mathcal E(\tau)$ with
$i\in I_p(\tau)$ and $k\notin I_p(\tau)$. Then $i$ has a neighbor with value $\neq p$ and by Step 2 it must leave
$p$ at time $\tau+1$. Consequently $|I_p(\tau+1)|<|I_p(\tau)|$.

Repeating this argument over successive windows, one shows that if $V(z(j))>0$ then within at most $\ell B$ steps
(at most $\ell$ possible decrements of vertex-occupancy), at least one vertex of $H(z(j))$ loses all agents and
cannot remain an extreme point of the later hull. Therefore there exists $\hat B:=\ell B$ such that
\[
H(z(j+\hat B))\subsetneq H(z(j))\qquad\text{whenever }V(z(j))>0.
\]
Strict containment of compact convex sets implies strict decrease of the diameter, hence
$V(z(j+\hat B))<V(z(j))$ whenever $V(z(j))>0$.

\emph{Step 4 (conclude $\limsup_{j \to \infty} V(z(j))=0$).}
If $\limsup_{j \to \infty} V(z(j))>0$, then along the subsequence $j_m:=m\hat B$ we would have a strictly decreasing bounded-below
sequence $V(z(j_m))$, contradicting convergence to a positive limit. Hence $\lim_{j\to\infty}V(z(j))=0$.
This is asymptotic consensus by Definition~\ref{def:asymp_consensus}.
\end{proof}
\section{Consensus via Distributed Multi-Step MPC}
\label{subsec:dmpc}

Here we represent the local optimization problem whose solution is the optimal control sequence which is applied by each local MPC controller.
For each agent $i$, at outer time $j$ we assign a local optimization control problem (OCP) as follows.

\begin{problem} \label{prob:local-OCP}
Consider a network of agents~\eqref{eq:agent_dyn} with an associated directed graph $(\mathcal{V},\mathcal{E}(j))$.
Assume that there exists an \afcl $q_i$ as in Definition~\ref{def:afcl}.
Assume that there exists a finite-step Lyapunov function $V$ associated with the decaying function $\alpha \in \Kinf$, $\alpha < \id$. Let positive definite matrices $Q$ $R$, and $\varepsilon \geq 0$ be given.
Also, at any $t = jM \in \Zp$ assume that the state $z(j) := x(jM)$ is known and the local barycenter $\bar{z}_i (j)$ as in~\eqref{eq:barycenter} is available to the agent $i$.
Compute the input sequence $u_i(0{:}M-1)$, which are the decision variables, and predicted state sequence $x_i(0{:}M)$ for the following optimization control problem (OCP), for each local agent $i = 1,\dots,\ell$.
\begin{equation}\label{eq:OCP}
\begin{aligned}
\min_{u_i(0{:}M-1)}\quad &
J_i := \sum_{k=0}^{M-1}\Big(\|x_i(k)-\bar z_i(j)\|_Q^2 + \|u_i(k)-u_i(k-1)\|_{R}^2\Big)\\
\text{s.t.}\quad
& x_i(0)=z_i(j),\\
& x_i(k+1)=g_i(x_i(k),u_i(k)),\qquad k=0,\dots,M-1,\\
& x_i(k)\in \X_i,\ u_i(k)\in \U_i,\qquad k=0,\dots,M-1,\\
& x_i(M)\in ,\\
& x_i(M)\in \X_i \cap \mathcal{C}_i (j) \\
& \| x_i(M)-\bar z_i(j) \| \leq \alpha (\| x_i(0)-\bar z_i(j) \|).
\end{aligned}
\end{equation} 
\end{problem}

The barycenter-tracking term is a \emph{local} stage objective encouraging movement toward local agreement.
Motivated by the sufficiency condition of the approximate strictness~\eqref{eq:A1b_ri_strictness}, the optimization problem aims to minimize the distance to the barycenter.
Motivated by the necessity of the convex hull inclusion~\eqref{eq:A1a_eps_hull}, this condition is encoded as the terminal constraint in the local OCP.
The feasibility of the inequality condition $\| x_i(M)-\bar z_i(j) \| \leq \alpha (\| x_i(0)-\bar z_i(j) \|)$ in~\eqref{eq:OCP} is ensured by Proposition~\ref{prop:fsCLF} thanks to the assumption on the exitence of \afcls $q_i$'s.

The distributed multi-step MPC algorithm is given as follows.
\begin{algorithm}[ht]
\caption{Distributed MPC routine}
\label{alg:dmproutine}
\begin{algorithmic}[1]
\State Fix $M\in\mathbb{N}$, $\varepsilon\ge 0$. Initialize $x_i(0)\in \X_i$.
\For{$j=0,1,2,\dots$}
  \State Each agent $i$ transmits $z_i(j)=x_i(jM)$ to neighbors and receives $\{z_k(j):k\in\mathcal{N}_i(j)\}$.
  \State Agent $i$ constructs $\mathcal{C}_i(j)$ and $\bar z_i(j)$.
  \State Agent $i$ solves the local OCP \eqref{eq:OCP} and obtains $u_i^\star(0{:}M-1)$.
  \State Apply $u_i(t)=u_i^\star(t-jM)$ for inner times $t=jM,\dots,jM+M-1$ (open loop).
\EndFor
\end{algorithmic}
\end{algorithm}
\subsection{Consensus under OCP \eqref{eq:OCP} via enforced finite-step barycenter contraction}
\label{subsec:asymp_consensus_ocp14}

The following result is a direct consequence of the necessary condition (N1) in Theorem~\ref{thm:moreau_nes_suf}.
This can be seen as a restatement of~\cite[Lemma~1]{Moreau.2005}.

\begin{assumption}[hull condition]\label{ass:eps_hull}
For all $i,j$,
\begin{equation}\label{eq:eps_hull}
\dist\big(z_i(j+1),\mathcal{C}_i(j)\big)\le 0,
\quad\text{equivalently}\quad z_i(j+1)\in \mathcal{C}_i(j) .
\end{equation}
\end{assumption}
The convexity inclusion in Assumption~\ref{ass:eps_hull} is guaranteed by the terminal constraint of each local OCP.
\begin{lemma}[One-step hull bound]\label{lem:one_step}
Under Assumption~\ref{ass:eps_hull},
\begin{equation}\label{eq:one_step_V}
V(z(j+1))\le V(z(j)) ,\qquad \forall j  \geq 0.
\end{equation}
\end{lemma}
\begin{proof}
Fix $j\ge0$ and recall $H(z(j)):=\co\{z_1(j),\dots,z_\ell(j)\}$.
For each $i$, we have by Assumption~4 that $z_i(j{+}1)\in \Cij$.
By definition of $\Cij$ in \eqref{eq:Ci_def_moreau}, $\Cij\subseteq H(z(j))$ because it is the convex hull of a subset
of $\{z_1(j),\dots,z_\ell(j)\}$.
Hence $z_i(j+1)\in H(z(j))$ for all $i$, and therefore
\[
H(z(j+1))=\co\{z_1(j{+}1),\dots,z_\ell(j{+}1)\}\subseteq H(z(j)).
\]
Taking diameters gives $\diam(H(z(j+1)))\le \diam(H(z(j)))$, i.e. $V(z(j+1))\le V(z(j))$.
\end{proof}

\begin{remark}\label{rem:need_for_strictness}
By Lemma~\ref{lem:one_step} we make sure that the convex hull $\Cij$ is invariant, the agents do not leave this set.
However, to show asymptotic consensus as in Definition~\ref{def:asymp_consensus} each agent shall expose a stronger property, namely lying in a relative interior (RI) of the convex hull $\Cij$, see Assumption~\ref{eq:A1b_ri_strictness}.
We note that the OCP~\eqref{eq:OCP} has no mechanism to enforce agents to satisfy the RI strictness~\ref{eq:A1b_ri_strictness}.
In fact, an agent could repeatedly move along a face or vertex of its neighbor hull, satisfying the contraction inequality, without ever entering the RI of the hull, while the convex-hull inclusion $z_i(j+1)=x_i(M)\in \Cij$ and the finite-step contraction toward the barycenter $\|z_i(j+1)-\bar z_i(j)\|\le \alpha(\|z_i(j)-\bar z_i(j)\|)$ are satisfied.
\end{remark}

\section{Impossibility of consensus under the current OCP formulation}\label{sec:impossibility}

This section shows that, under the current OCP formulation~\eqref{eq:OCP}, asymptotic consensus
cannot be guaranteed without an additional strictness assumption which excludes boundary-invariant trajectories. 
In particular, the RI strictness~\eqref{eq:A1b_ri_strictness} can be concluded from the current setting of OCP~\eqref{eq:OCP}.

\begin{proposition}[Impossibility of concluding RI strictness (10) from under OCP \eqref{eq:OCP}]\label{prop:impossibility}
Consider the closed-loop multi-agent system generated by Algorithm~\ref{alg:dmproutine} based on OCP~\eqref{eq:OCP}. Assume:
\begin{enumerate}[(i)]
\item OCP \eqref{eq:OCP} is feasible for all agents $i \in \V$ and all times $j\ge 0$; in particular, the terminal convex-hull constraint
$
z_i(j+1)\in \Cij
$ 
and the terminal finite-step inequality
$
\|z_i(j+1)-\bar z_i(j)\|\le \alpha(\|z_i(j)-\bar z_i(j)\|)$, $\alpha < \id$ hold for all $i \in \V,j \geq 0$;
\item the communication graph satisfies Assumption~\ref{ass:joint_tree} (joint spanning-tree connectivity).
\end{enumerate}
Then OCP~\eqref{eq:OCP} does \emph{not} enforce Assumption~\ref{ass:asymp_moreau_A1}(ii)~\eqref{eq:A1b_ri_strictness}.
More precisely, there exist a finite network, a jointly connected graph,
and explicit agent states such that:
\begin{enumerate}[(a)]
\item for some time $j$ and some agent $i$, the local set $\Cij$ is non-singleton,
\item OCP~\eqref{eq:OCP} for that agent admits a feasible solution whose implemented terminal state
satisfies $z_i(j+1)\in \rb\Cij$,
\item hence, one cannot deduce
\(
C_i(j)\ \text{non-singleton} \Rightarrow z_i(j+1)\in\ri(C_i(j))
\)
from OCP~(11).
\end{enumerate}
\end{proposition}
\begin{remark}
A direct consequence of Proposition~\ref{prop:impossibility} is that asymptotic consensus cannot be \emph{guaranteed} from OCP~\eqref{eq:OCP} and connectivity
alone, since the RI strictness condition required by Theorem~\ref{thm:moreau_nes_suf} is not
implied by OCP~\eqref{eq:OCP}.
We note that as stated by Theorem~\ref{thm:moreau_nes_suf} the RI strictness is not necessary as a property of the update map,
but, as seen from Proposition~\ref{prop:impossibility}, some form of strictness is necessary as a property of the closed-loop trajectories to exclude trapping in a boundary invariance.
\end{remark}
\begin{proof}[Proof of Proposition~\ref{prop:impossibility}]
Consider $\ell=4$ agents in $\R^2$ with samples at some time $j=0$ given by
\[
z_1(0)=(0,0),\quad z_2(0)=(0,1),\quad z_3(0)=(1,0),\quad z_4(0)=(1,1).
\]
Let the communication graph be time-invariant and complete:
\[
\Ni(0)=\{1,2,3,4\}\setminus\{i\}\qquad\forall i\in\{1,2,3,4\}.
\]
This graph is strongly connected and satisfies any joint connectivity condition trivially.

For agent $i=1$ we then have
\[
\C_1(0)=\co\{z_1(0),z_2(0),z_3(0),z_4(0)\}=[0,1]\times[0,1],
\]
which is non-singleton and full-dimensional.
Moreover, the barycenter for agent $1$ is
\[
\bar z_1(0)=\frac{1}{4}\bigl(z_1(0)+z_2(0)+z_3(0)+z_4(0)\bigr)=(0.5,0.5)\in \ri(\C_1(0)).
\]
Take $\alpha(r)=0.9r$.
We now show there exists a feasible terminal point \emph{on the boundary} of $\C_1(0)$ satisfying the terminal inequality.

Let agent~1 implement the terminal state
\[
z_1(1):=(0,0.5).
\]
Then $z_1(1)\in\partial([0,1]^2)$ (it lies on the left edge), hence $z_1(1)\notin \ri(\C_1(0))$.
Compute the distances to the barycenter:
\[
\|z_1(0)-\bar z_1(0)\|=\|(0,0)-(0.5,0.5)\|=\sqrt{0.5},
\]
\[
\|z_1(1)-\bar z_1(0)\|=\|(0,0.5)-(0.5,0.5)\|=0.5.
\]
Hence the terminal inequality holds:
\[
\|z_1(1)-\bar z_1(0)\| = 0.5 \le 0.9\sqrt{0.5}=\alpha(\|z_1(0)-\bar z_1(0)\|),
\]
since $0.9\sqrt{0.5}\approx 0.6364$.

It remains only to note that this choice is compatible with OCP~\eqref{eq:OCP}'s dynamics and input constraints.
For instance, take a single-integrator prediction model for agent~1 over horizon
$M=1$,
\[
x_1(1)=x_1(0)+u_1(0),
\]
with admissible inputs $u_1(0)\in \U_1:=[-1,1]^2$.
Then $x_1(0)=z_1(0)=(0,0)$ and the input $u_1(0)=(0,0.5)$ yields $x_1(1)=(0,0.5)=z_1(1)$,
so the terminal state $z_1(1)$ is dynamically feasible.
Therefore OCP~\eqref{eq:OCP} admits a feasible solution for agent~1 whose implemented terminal
state satisfies
$z_1(1)\in\partial \C_1(0)$ while $\C_1(0)$ is non-singleton.
Hence, Condition~\eqref{eq:A1b_ri_strictness} is not implied by OCP~\eqref{eq:OCP}, which proves the claim.
\end{proof}
\section{DMPC with Lexicographic Selection}\label{sec:OCP_with_lex_sel}
The terminal constraints in OCP~\eqref{eq:OCP} guarantee the hull inclusion \eqref{eq:eps_hull}, but do not,
in general, force $x_i(M)$ to lie in $\ri(\Cij)$ whenever disagreement is
nonzero. 
Indeed, boundary points of $\Cij$ may remain feasible and even optimal.

In this section, we propose a \emph{minimal} addition to OCP~\eqref{eq:OCP} without changing its setting, hence, without touching the feasibility set.
To separate feasibility from implementation, we introduce the set of terminal points that can actually be realized under OCP~\eqref{eq:OCP}.
Then we just introduce a \emph{lexicographic selection} rule that prefers interior points when the optimizer is non-unique.

Fix an outer time $j\in\mathbb N_0$ and agent $i\in V$.
Let $x_i^\star(0{:}M-1)$ denote the state trajectories of agents~\eqref{eq:agent_dyn} under optimizers $u_i^\star(0{:}M-1)$ of OCP~\eqref{eq:OCP}.

The set of feasible terminal points under OCP~\eqref{eq:OCP} is defined as follows.

\begin{definition}[Feasible terminal set]\label{def:Fi}
For each $i$ and $j$, define the \emph{feasible terminal set} induced by OCP~\eqref{eq:OCP} as
\[
F_i(j) \;:=\;\Big\{\, \xi\in\mathbb R^d \ \Big|\ \exists u_i^\star(0{:}M-1) \,\,\,\mathrm{ s.t. }\,\,\, x_i^\star(M)=\xi \,\Big\}.
\]
\end{definition}
We note that, by construction, $F_i(j)\subseteq \Cij$.

\subsection{When OCP~\eqref{eq:OCP} does not automatically yield RI strictness}
\label{subsec:ri_selection}
The question ``what if OCP finds no solution such that $z_i(j{+}1)\in\ri(\Cij)$?'' splits into
three logically distinct cases.

\paragraph{Three possibilities.}
Let $\C_i(j)$ be the local hull at time $j$.
\begin{enumerate}
\item \textbf{Case A (singleton hull).} If $\Cij=\{p\}$, then $\ri(\Cij)=\{p\}$ and RI
strictness is automatic.
This occurs exactly when all points in
$\{z_i(j)\}\cup\{z_k(j):k\in N_i(j)\}$ coincide.\\
\emph{How to handle it:} no action needed; the update is already consistent with consensus.
\item \textbf{Case B ($\ri(\Cij)\neq\emptyset$ but \emph{not} chosen by primary optimum).} 
This is the typical situation with linear/convex objectives: the set of optimal terminal points
may lie on a face of $\Cij$. In this case, there \emph{exist} feasible interior points
(e.g.\ $\ri(\Cij)$ is nonempty whenever $\Cij$ has more than one point), but the primary objective
does not select them.\\
\emph{How to handle it:} use a \emph{lexicographic} (two-stage) selection rule that \emph{preserves} the feasible set and the primary optimum value, but selects an interior point
whenever possible. 
This is formalized below, see Subsection~\ref{subsec:lexicographic-selc}.
\item \textbf{Case C ($F_i(j)\cap \ri(\Cij)=\emptyset$)} 
In this case, consensus may still occur in special instances, but cannot be guaranteed in general from OCP~\eqref{eq:OCP} alone.\\
\emph{How to handle it:}
We divide this case into two sub-cases:
\begin{itemize}
\item If Case C occurs only finitely many times (or not on the ``extreme'' agents responsible for the
global hull diameter), consensus may still be provable under additional \emph{trajectory} conditions.
However, such conditions are extra assumptions and generally hard to verify from OCP~\eqref{eq:OCP}.
\item If Case C can occur infinitely often under persistent disagreement, then no RI strictness
argument can be guaranteed.
To rule it out, one must modify the formulation (e.g.\ soften constraints, change
terminal condition, or add an explicit interiority-promoting mechanism).
Any such modification changes either the feasible set or the implemented selection, and thus goes beyond OCP~\eqref{eq:OCP} as stated.
\end{itemize} 
\end{enumerate}

\subsection{Lexicographic (two-stage) selection among optimal solutions.}\label{subsec:lexicographic-selc}

We now formalize the weakest modification that does \emph{not} shrink the feasible set of OCP~\eqref{eq:OCP}:
keep all constraints and the primary objective unchanged, but define a secondary selection among primary optimizers.

\begin{assumption}[RI existence on feasible terminal set]\label{ass:RIexist_ocp}
Whenever $\Cij$ is not a singleton, the feasible terminal set intersects its RI:
\begin{equation}\label{eq:RIexist_ocp}
F_i(j)\cap \ri(\Cij)\neq\emptyset.
\end{equation}
\end{assumption}

\begin{remark}[Assumption~\ref{ass:RIexist_ocp} is design-neutral and minimal]\label{rem:design_neutral}
Assumption~\ref{ass:RIexist_ocp} is \emph{necessary} for any implementation that aims to realize
$x_i(M)\in\ri(\Cij)$ while keeping the constraints of OCP~\eqref{eq:OCP} unchanged:
if $F_i(j)\cap \ri(\Cij)=\emptyset$, then RI strictness is impossible at that $(i,j)$ without
changing the constraint set. The assumption does not impose margins, weights, or quantitative strictness; it
only asserts existence of at least one feasible RI terminal point.
\end{remark}

Let $J_i$ be the primary cost in OCP~\eqref{eq:OCP}.
At each $(i,j)$:
\begin{enumerate}
\item Solve OCP~\eqref{eq:OCP} to obtain the optimal value $J_i^\star(j)$.
\item Among all feasible solutions with $J_i\le J_i^\star(j)$, select $x_i(M)$ that maximizes an
interiority measure, e.g., see~\eqref{eq:interiority-measure}.
\end{enumerate}
A convenient interiority measure that can be stated abstractly (without committing to a particular
polytope representation) is
\begin{equation}\label{eq:interiority-measure}
    \phi_{i,j}(\xi) := \dist(\xi,\partial \Cij).
\end{equation}
Then the secondary problem is
\begin{equation}\label{eq:secondary}
\max_{\xi\in F_i(j)}\ \phi_{i,j}(\xi)\qquad \text{s.t.}\quad \exists u_i^\star(0{:}M-1) \Rightarrow 
x_i(M)=\xi \text{ and } J_i\le J_i^\star(j).
\end{equation}
Any maximizer $\xi^\dagger$ of \eqref{eq:secondary} is selected, and the implemented update is
$z_i(j{+}1)=\xi^\dagger$.
This preserves the original optimum value and only fixes a tie-breaking rule.

\begin{remark}[No feasibility shrinkage]\label{rem:no_feas_shrink_ocp}
The lexicographic selection rule does \emph{not} shrink the feasibility set of OCP~\eqref{eq:OCP} and does not
add any additional constraint. It only specifies which solution is implemented when multiple primary optimizers
exist.
\end{remark}

\begin{lemma}[Lexicographic selection yields RI strictness whenever feasible]\label{lem:lex_RI}
Assume $\Cij$ is not a singleton and Assumption~\ref{ass:RIexist_ocp} holds at $(i,j)$.
Then the lexicographic selection produces $z_i(j+1)=x_i(M)\in \ri(\Cij)$.
\end{lemma}

\begin{proof}
Since $\Cij$ is compact and convex, $\partial \Cij$ is closed and the function
$\phi_{i,j}(\xi)=\dist(\xi,\partial \Cij)$ is continuous.
For any $\xi\in \partial \Cij$ we have $\phi_{i,j}(\xi)=0$, while for any
$\xi\in\ri(\Cij)$ we have $\phi_{i,j}(\xi)>0$.
By Assumption~\ref{ass:RIexist_ocp} there exists a feasible $\tilde\xi\in F_i(j)\cap\ri(\Cij)$.
Hence the maximum of $\phi_{i,j}$ over the feasible primary-optimal set is strictly positive, and any maximizer
$\xi^\dagger$ satisfies $\phi_{i,j}(\xi^\dagger)>0$, thus $\xi^\dagger\notin\partial \Cij$.
Since $\xi^\dagger\in F_i(j)\subseteq \Cij$, we conclude $\xi^\dagger\in \ri(\Cij)$.
\end{proof}

\subsection{Consensus theorem under OCP~\eqref{eq:OCP} with lexicographic RI selection}

\begin{theorem}[Consensus under OCP~\eqref{eq:OCP} with minimal tie-breaking]\label{thm:consensus_ocp_lex}
Consider the distributed multi-step MPC scheme based on OCP~\eqref{eq:OCP}. Assume:
\begin{enumerate}
\item \textbf{(Feasibility)} OCP~\eqref{eq:OCP} is feasible for all agents $i$ and all times $j\ge0$.
\item \textbf{(Joint connectivity)} Assumption~\ref{ass:joint_tree} holds.
\item \textbf{(RI existence)} Assumption~\ref{ass:RIexist_ocp} holds.
\item \textbf{(Implementation)} Each agent implements the lexicographic selection rule described above.
\end{enumerate}
Then the implemented closed-loop updates satisfy Assumption~\ref{ass:asymp_moreau_A1}, and consequently the
network achieves asymptotic consensus in the sense of Definition~\ref{def:asymp_consensus}.

\smallskip
\noindent\textbf{Necessity statement (design-neutral).}
Under locality of the implemented update law, the hull inclusion \eqref{eq:eps_hull} together with
Assumption~\ref{ass:joint_tree} is necessary for consensus for all initial conditions, as stated in the
necessity part of Theorem~\ref{thm:moreau_nes_suf}.
\end{theorem}

\begin{proof}
Hull inclusion \eqref{eq:eps_hull} holds because it is enforced as the terminal constraint of OCP~\eqref{eq:OCP}.
If $\Cij$ is a singleton then $z_i(j+1)\in \ri(\Cij)$ trivially. Otherwise, by
Assumption~\ref{ass:RIexist_ocp} and Lemma~\ref{lem:lex_RI}, the implemented terminal point lies in
$\ri(\Cij)$. Hence Assumption~\ref{ass:asymp_moreau_A1} holds for the implemented closed loop.
With Assumption~\ref{ass:joint_tree}, the sufficiency part of Theorem~\ref{thm:moreau_nes_suf} yields asymptotic
consensus.
The necessity statement is exactly the necessity part of Theorem~\ref{thm:moreau_nes_suf}.
\end{proof}

\begin{remark}[What if Assumption~\ref{ass:RIexist_ocp} fails?]\label{rem:if_RIexist_fails}
If for some $(i,j)$ one has $F_i(j)\cap \ri(C_i(j))=\emptyset$, then RI strictness cannot be achieved at
that time by any selection rule without modifying constraints. Consensus may still occur in specific
instances, but it cannot be guaranteed in general by geometric-type strictness arguments as in~\eqref{eq:A1b_ri_strictness}.
In that case, one must either:
(i) accept an additional trajectory-level assumption (e.g.\ strictness occurs infinitely often on extreme
faces), or (ii) modify the MPC formulation (e.g.\ soften constraints or add an interiority-promoting term).
\end{remark}

\subsection{Distributed multi-step MPC algorithm with lexicographic selection}
Now we present our final distributed multi-step MPC algorithm with lexicographic selection.
Before proceeding further, we raise a remark aiming to numerically enhance the robustness of the proposed algorithm.

\begin{remark}[On numerical implementation of the lexicographic selection]\label{rem:lex_tolerance}
The lexicographic (secondary) problem \eqref{eq:secondary} is formulated with the exact primary-optimality constraint \(J_i \le J_i^\star(j)\).
All subsequent results, in particular Lemma~\ref{lem:lex_RI} and Theorem~\ref{thm:consensus_ocp_lex}, are therefore stated in terms of the \emph{primary-optimal set} of OCP~\eqref{eq:OCP}.

In numerical implementations, however, optimal control problems are solved in finite precision, and solver tolerances may lead to solutions whose objective value differs slightly from \(J_i^\star(j)\). In practice, this is often handled by enforcing the primary-optimality constraint in \eqref{eq:secondary} up to numerical tolerances, or
equivalently by allowing a very small slack \(\delta_{\mathrm{lex}}>0\) and considering the \(\delta_{\mathrm{lex}}\)-suboptimal set
\[
\{\, J_i \le J_i^\star(j) + \delta_{\mathrm{lex}} \,\}.
\]
If such a tolerance is introduced, the implemented terminal state may be only \emph{near-optimal} rather than exactly primary-optimal, and Lemma~\ref{lem:lex_RI} and Theorem~\ref{thm:consensus_ocp_lex} would need to be reworded accordingly, replacing the primary-optimal set by the \(\delta_{\mathrm{lex}}\)-suboptimal set.

Since this modification does not affect the conceptual role of the lexicographic selection but would unnecessarily complicate the theoretical exposition, we do not consider it explicitly in this paper. Nevertheless, the introduction of a small \(\delta_{\mathrm{lex}}>0\) is natural and often useful in numerical implementations.
\end{remark}

Now we integrate the lexicographic selection into our original distributed multi-step MPC algorithm Algorithm~\ref{alg:dmproutine} to summarize our final controller design.
Note that we consider the optional, practically useful lexicographic tolerance $\delta_{\mathrm{lex}}$ according to Remark~\ref{rem:lex_tolerance}.

\begin{algorithm}[ht]
\caption{Distributed multi-step MPC routine with lexicographic selection}
\label{alg:dmproutine_lex}
\begin{algorithmic}[1]
\State Fix $M\in\N$, $\varepsilon\ge 0$ and a lexicographic tolerance $\delta_{\mathrm{lex}} \geq 0$.
\State Initialize $x_i(0)\in \X_i$ for all agents $i$.
\For{$j=0,1,2,\dots$}
  \State Each agent $i$ transmits $z_i(j)=x_i(jM)$ to neighbors and receives $\{z_k(j):k\in\mathcal{N}_i(j)\}$.
  \State Agent $i$ constructs the local convex hull $\mathcal{C}_i(j)$ and the local barycenter $\bar z_i(j)$.
  \State Agent $i$ solves the local OCP \eqref{eq:OCP} (primary problem) and obtains an optimal input sequence
        $u_i^{\star}(0{:}M{-}1)$ and the optimal value $J_i^{\star}(j)$.

  \State \textbf{Lexicographic selection (secondary problem).}
  \Statex \hspace{1.5em} Among all feasible solutions of OCP \eqref{eq:OCP} whose primary cost satisfies
  \Statex \hspace{1.5em} $J_i \le J_i^{\star}(j) + \delta_{\mathrm{lex}}$,
  \Statex \hspace{1.5em} agent $i$ selects a solution that maximizes the interiority measure
  \Statex \hspace{1.5em} $\phi_{i,j}(x_i(M)):=\dist\!\big(x_i(M),\partial \mathcal{C}_i(j)\big)$.
  \Statex \hspace{1.5em} Denote the resulting lexicographically selected optimal input by
  \Statex \hspace{1.5em} $u_i^{\dagger}(0{:}M{-}1)$ and its associated terminal state by $x_i^{\dagger}(M)$.

  \State Apply $u_i(t)=u_i^{\dagger}(t-jM)$ for inner times $t=jM,\dots,jM+M-1$ (open loop).
\EndFor
\end{algorithmic}
\end{algorithm}

\section{Application to Agents with Linear Dynamics}\label{sec:linear-dynamics}

We briefly specialize the discussion to heterogeneous discrete-time linear agents
\begin{equation}\label{eq:lti_agents}
x_i(t{+}1)=A_i x_i(t)+B_i u_i(t),\qquad x_i(t)\in \X_i,\;u_i(t)\in \U_i,
\end{equation}
where $\X_i\subset\R^d$ and $\U_i\subset\R^{m_i}$ are nonempty, convex, and compact.

\begin{lemma}[Open-loop reachability ball in $M$ steps]\label{lem:reach_ball}
Fix $i$ and $M\in\N$. Define the $M$-step controllability matrix
\[
C_{i,M}:=\big[B_i\;\;A_iB_i\;\;\dots\;\;A_i^{M-1}B_i\big]\in\R^{d\times (m_iM)}.
\]
Assume $\rank(C_{i,M})=d$ and that $\U_i$ contains a Euclidean ball around the origin,
i.e., there exists $r_{u,i}>0$ such that $B_{r_{u,i}}\subseteq \U_i$.
Then for every initial state $\xi\in\R^d$ the set of states reachable in $M$ steps under inputs
$u(0{:}M{-}1)\in \U_i^M$ contains a Euclidean ball:
\[
\B_{\rho_i}(A_i^M\xi)\subseteq \Big\{\,x_i(M)\,:\, x_i(0)=\xi,\; u(0{:}M{-}1)\in \U_i^M\,\Big\},
\]
where one may take $\rho_i := \sigma_{\min}(C_{i,M})\,r_{u,i}>0$ and
$\sigma_{\min}(C_{i,M})$ denotes the smallest singular value.
\end{lemma}

\begin{proof}
Stack the inputs into $U:=(u(0)^\top,\dots,u(M{-}1)^\top)^\top\in\R^{m_iM}$. The $M$-step
state satisfies the standard lifting relation
\[
x_i(M)=A_i^M\xi+C_{i,M}U.
\]
Because $\B_{r_{u,i}}\subseteq \U_i$, we have $\B_{r_{u,i}}^M\subseteq \U_i^M$, and hence
$\B_{r_{u,i}}^{m_iM}\subseteq \{U: u(k)\in \U_i\}$ (up to identifying the product of balls with a ball
in the stacked space). Therefore $C_{i,M}\B_{r_{u,i}}^{m_iM}$ is contained in the $M$-step
reachable displacement set. Since $\rank(C_{i,M})=d$, the linear map $C_{i,M}$
is onto and satisfies
\[
C_{i,M}\B_{r_{u,i}}^{m_iM}\supseteq \B_{\sigma_{\min}(C_{i,M})r_{u,i}}^{d}
= \B_{\rho_i},
\]
which implies $\B_{\rho_i}(A_i^M\xi)\subseteq \{x_i(M)\}$ as claimed.
\end{proof}

\medskip

Define, as in Definition~\ref{def:Fi}, the feasible terminal set induced by OCP~\eqref{eq:OCP},
\[
F_i(j):=\big\{\xi\in\R^d:\exists\,u_i^\star(0{:}M{-}1)\ \text{s.t.}\ x_i^\star(M)=\xi\big\},
\]
and recall $F_i(j)\subseteq \Cij$.

\begin{corollary}[Explicit LTI conditions implying feasibility and Assumption~\ref{ass:RIexist_ocp}]\label{cor:lti_feas_ri}
Consider the distributed multi-step MPC based on OCP~\eqref{eq:OCP} for agents \eqref{eq:lti_agents}.
Fix $M\in\N$ and suppose that for each agent $i$:
\begin{enumerate}
\item[(L1)] $\rank(C_{i,M})=d$ and $\B_{r_{u,i}}\subseteq \U_i$ for some $r_{u,i}>0$.
\item[(L2)] (Slack state constraints along the $M$-step steering) For every outer time $j$ and every
agent $i$, there exists at least one $M$-step admissible input sequence that keeps the predicted
state within $\X_i$ and steers the endpoint into
\[
\B_{\rho_i}(\bar z_i(j))\cap \Ci(j),
\qquad \rho_i=\sigma_{\min}(C_{i,M})\,r_{u,i}.
\]
(For instance, (L2) holds whenever $\X_i$ is sufficiently large so that the straight-line steering
inputs constructed from Lemma~\ref{lem:reach_ball} remain in $\X_i$ for all relevant configurations.)
\item[(L3)] (Nondegenerate hull when disagreement exists) Whenever $\Ci(j)$ is not a singleton,
$\ri(\Ci(j))\neq\emptyset$ (equivalently: $\Ci(j)$ has positive dimension in its affine hull).
\end{enumerate}
Then the following hold for every $(i,j)$ with $\Ci(j)$ not a singleton:
\begin{enumerate}
\item[(a)] OCP~\eqref{eq:OCP} is feasible at $(i,j)$; in particular, the terminal inequality
$\|z_i(j{+}1)-\bar z_i(j)\|\le \alpha(\|z_i(j)-\bar z_i(j)\|)$ can be satisfied.
\item[(b)] Assumption~\ref{ass:RIexist_ocp} holds at $(i,j)$, i.e. $F_i(j)\cap \ri(\Ci(j))\neq\emptyset$.
\end{enumerate}
Consequently, under Assumption~\ref{ass:joint_tree} and the lexicographic implementation as in Algorithm~\ref{alg:dmproutine_lex} with $\delta_{\mathrm{lex}} = 0$,
Theorem~\ref{thm:consensus_ocp_lex} applies and yields asymptotic consensus.
\end{corollary}

\begin{proof}
Fix $(i,j)$ and abbreviate $\C:=\Cij$ and $\bar z:=\bar z_i(j)$.
By (L2), there exists an admissible $M$-step input sequence producing an endpoint
$\xi\in \B_{\rho_i}(\bar z)\cap \C$. In particular, $\xi\in \C$, so the terminal hull constraint of OCP~\eqref{eq:OCP}
is satisfied. Moreover, since $\xi\in \B_{\rho_i}(\bar z)$, we can choose in (L2) either $\xi=\bar z$ (if
$\bar z$ itself is reachable under the admissibility requirement) or otherwise a point $\xi$ arbitrarily
close to $\bar z$. In both cases the terminal contraction inequality of OCP~\eqref{eq:OCP} is satisfied because
\[
\|\xi-\bar z\| \le \alpha(\|z_i(j)-\bar z\|),
\]
holds trivially when $\xi=\bar z$ (left-hand side $=0$) and can be ensured when $\xi$ is sufficiently
close to $\bar z$ by continuity and monotonicity of $\alpha<\id$. This proves feasibility (a).

To show (b), assume $\C$ is not a singleton. By (L3), $\ri(\C)\neq\emptyset$. Since $\bar z$ is a convex
combination of points in $\C$ (by definition of the barycenter), it belongs to $\C$. If $\bar z\in\ri(\C)$,
then $\B_{\rho_i}(\bar z)\cap \ri(\C)\neq\emptyset$ for all sufficiently small radii, and by (L2) we may
choose $\xi\in \B_{\rho_i}(\bar z)\cap \ri(\C)\subseteq \C$, hence $\xi\in F_i(j)\cap\ri(\C)$.
If $\bar z\notin\ri(\C)$, then (L2) still guarantees existence of an admissible endpoint in
$\B_{\rho_i}(\bar z)\cap \C$; by (L3) and basic convex geometry, any neighborhood (in the affine hull)
of a non-singleton convex set intersects its RI, hence we can refine the choice in (L2)
to obtain an admissible endpoint $\xi\in \B_{\rho_i}(\bar z)\cap \ri(\C)$. In either case,
$F_i(j)\cap\ri(\C)\neq\emptyset$, which is Assumption~\ref{ass:RIexist_ocp}.

The final claim follows by invoking Lemma~\ref{lem:lex_RI} and then Theorem~\ref{thm:consensus_ocp_lex}.
\end{proof}

\begin{remark}[Interpretation and a simple sufficient special case]\label{rem:lti_interpretation}
Condition (L1) is an explicit, checkable $M$-step controllability condition.
Condition (L2) captures precisely the informal requirement that \emph{constraints are not too tight}:
there must exist at least one admissible $M$-step steering that reaches a neighborhood of the
local barycenter while remaining inside $\X_i$ and $\U_i$.
A particularly transparent sufficient case is the following: if $\X_i=\R^d$ (or $\X_i$ is a large convex
set containing all predicted trajectories of interest), and $\B_{r_{u,i}}\subseteq \U_i$ with
$\rank(C_{i,M})=d$, then Lemma~\ref{lem:reach_ball} implies that every point in a ball around
$A_i^M z_i(j)$ is reachable in $M$ steps. If, in addition, $\bar z_i(j)$ lies within that reachable ball
and $\Cij$ is non-singleton, then OCP~\eqref{eq:OCP} is feasible and Assumption~\ref{ass:RIexist_ocp} holds automatically.
\end{remark}

\subsection{Explicit special cases: single- and double-integrators}\label{subsec:explicit_integrators}

Throughout this subsection, we focus on single- and double-integrators and assume that only input constraints are present.
This is exactly the regime in which one can give closed-form feasibility and horizon bounds. 
The results can be extended to additional state constraints by replacing the explicit bounds
below with an a~priori invariant ``tube'' condition as in Corollary~\ref{cor:lti_feas_ri}.

\medskip

\begin{corollary}[Single-integrators: explicit horizon and automatic RI existence]\label{cor:single_integrator_explicit}
Consider agents with discrete-time single-integrator dynamics
\begin{equation}\label{eq:si}
x_i(t{+}1)=x_i(t)+u_i(t),\qquad u_i(t)\in \U_i:=\{u\in\R^d:\|u\|\le u_{i,\max}\},
\end{equation}
and assume no additional state constraints along the prediction horizon of OCP~\eqref{eq:OCP}.
Fix an outer time $j$ and let $z_i(j)=x_i(jM)$ and $\bar z_i(j)\in \Cij$ be the local barycenter as in~\eqref{eq:barycenter}.

Recall the global diameter $V$ as in~\eqref{eq:V_diameter}.
Let $u_{\min} := min_{i\in\V} u_{i,\max}$ and assume that $u_{\min}>0$ for all agents.

Let the terminal inequality in OCP~\eqref{eq:OCP} be given by
\[
\|x_i(M)-\bar z_i(j)\|\le \alpha(\|z_i(j)-\bar z_i(j)\|),\qquad \alpha<\id.
\]
Then the following hold.

\begin{enumerate}
\item[(a)] \textbf{Explicit feasibility ($M$-step reachability).}
If
\begin{equation}\label{eq:M_SI}
M \ \ge\ \Big\lceil \frac{V(z(0))}{u_{\min}}\Big\rceil,
\end{equation}
then for every $i$ and every $j\ge 0$, OCP~\eqref{eq:OCP} is feasible for agent $i\in\V$ at time $j$.

\item[(b)] \textbf{Automatic RI existence.}
If $\Cij$ is not a singleton, then Assumption~\ref{ass:RIexist_ocp} holds for all $i\in\V,j\geq 0$.

\item[(c)] \textbf{Consensus under lexicographic selection.}
Under joint connectivity (Assumption~\ref{ass:joint_tree}) and the lexicographic implementation, Theorem~\ref{thm:consensus_ocp_lex} applies and yields asymptotic consensus.
\end{enumerate}

\end{corollary}

\begin{proof}
Fix $i$ and $j$. For single-integrators, any endpoint $\xi$ can be reached in $M$ steps by choosing a
constant input
\[
u_i(t)=\frac{\xi-z_i(j)}{M},\qquad t=jM,\dots,jM+M-1,
\]
provided $\|u_i(t)\|\le u_{i,\max}$, equivalently $\|\xi-z_i(j)\|\le M u_{i,\max}$.

\emph{Step 1: feasibility.} Choose the terminal point $\xi=\bar z_i(j)$. Since $\bar z_i(j)\in \Cij$,
\[
\|\bar z_i(j)-z_i(j)\|\ \le\ \diam(\Cij)\ \le\ V(z(j))\ \le\ V(z(0)),
\]
where $V(z(j))\le V(z(0))$ follows from iterative application of Lemma~\ref{lem:one_step}.
Thus, if $M u_{i,\max}\ge V(z(0))$ (and hence if \eqref{eq:M_SI} holds), then $\bar z_i(j)$ is reachable in
$M$ steps under the input bound. With $\xi=\bar z_i(j)$ the terminal inequality holds trivially because
$\|\xi-\bar z_i(j)\|=0\le \alpha(\|z_i(j)-\bar z_i(j)\|)$, and the terminal hull constraint $\xi\in \Cij$
also holds. Hence OCP~\eqref{eq:OCP} is feasible for all $(i,j)$.

\emph{Step 2: RI existence.} 
If $\Cij$ is not a singleton, then $\ri(\Cij)\neq\emptyset$. For any
convex hull of finitely many points, the barycenter with strictly positive weights belongs to the
RI of that hull (relative to its affine hull). In particular, the barycenter $\bar z_i(j)\in\ri(\Cij)$ whenever not all points in its defining set coincide. 
Therefore, the reachable terminal point $\xi=\bar z_i(j)$ constructed above satisfies $\xi\in F_i(j)\cap\ri(\Cij)$,
consequently Assumption~\ref{ass:RIexist_ocp} holds.

\emph{Step 3:} Now Lemma~\ref{lem:lex_RI} yields the RI strictness under the lexicographic selection whenever
Assumption~\ref{ass:RIexist_ocp} holds, and Theorem~\ref{thm:consensus_ocp_lex} concludes consensus under Assumption~\ref{ass:joint_tree}. 
\end{proof}

\medskip

\begin{corollary}[Double-integrators: explicit horizon and automatic RI existence]\label{cor:double_integrator_explicit}
Consider agents with discrete-time double-integrator dynamics
\begin{equation}\label{eq:di}
r_i(t{+}1)=r_i(t)+v_i(t),\qquad v_i(t{+}1)=v_i(t)+u_i(t),
\qquad \|u_i(t)\|\le u_{i,\max},
\end{equation}
and assume no additional state constraints along the prediction horizon of OCP~\eqref{eq:OCP}.
Let the consensus variable be the \emph{position} $r_i$.

Assume that OCP~\eqref{eq:OCP} is posed on the lifted state $x_i=(r_i,v_i)$ but its terminal hull constraint
is imposed on positions, i.e.
\[
r_i(jM)\in \C_i(j):=\co(\{r_i(jM)\}\cup\{r_k(jM):k\in\mathcal N_i(j)\}),
\]
and the terminal inequality contracts toward the local barycenter $\bar r_i(j)$ of that hull:
\[
\|r_i(jM)-\bar r_i(j)\|\le \alpha(\|r_i(jM)-\bar r_i(j)\|),\qquad \alpha<\id.
\]
Define the outer-time position diameter $V_r(r(j)):=\diam(H(r(jM))$.
Let $u_{\min} := min_{i\in\V} u_{i,\max}$ and assume that $u_{\min}>0$ for all agents.
Define bounds on the initial outer-time velocities
\[
v_{\max}:=\max_{i}\|v_i(0)\|.
\]
Choose $M$ such that
\begin{equation}\label{eq:M_DI}
M\ \ge\ \max\!\left\{2,\ \Big\lceil \frac{2v_{\max}}{u_{\min}}\Big\rceil
\ +\ 2\Big\lceil \sqrt{\frac{V_r(0)}{u_{\min}}}\Big\rceil\right\}.
\end{equation}
Then the following hold.
\begin{enumerate}
\item[(a)] \textbf{Explicit feasibility.} For every $i$ and every $j\ge 0$, OCP~\eqref{eq:OCP} is feasible at $(i,j)$.
\item[(b)] \textbf{Automatic RI existence.} If $\C_i(j)$ is not a singleton, then $F_i(j)\cap\ri(\C_i(j))\neq\emptyset$.
\item[(c)] \textbf{Consequently (consensus under lexicographic selection).}
Under joint connectivity (Assumption~\ref{ass:joint_tree}) and lexicographic implementation, Theorem~\ref{thm:consensus_ocp_lex} yields asymptotic consensus in positions.
\end{enumerate}
\end{corollary}

\begin{proof}
We construct an explicit admissible $M$-step input that reaches the barycenter in position.

\emph{Step 1: velocity reset.}
Over $M_1:=\lceil 2v_{\max}/u_{\min}\rceil$ steps apply constant acceleration
$u_i(t)=-v_i(jM)/M_1$. Then $\|u_i(t)\|\le u_{\min}\le u_{i,\max}$ and $v_i(jM+M_1)=0$.

\emph{Step 2: rest-to-rest translation.}
Starting from $v=0$, a standard symmetric accelerate/decelerate maneuver of length $2M_2$
can move the position by any vector $\Delta r$ satisfying $\|\Delta r\|\le M_2^2 u_{i,\max}$:
accelerate with $u=\Delta r/M_2^2$ for $M_2$ steps, $M_2$ is specified below, and then apply $u=-\Delta r/M_2^2$ for
$M_2$ steps. The net velocity returns to zero and the net displacement is exactly $\Delta r$.

Set the target displacement $\Delta r:=\bar r_i(j)-r_i(jM+M_1)$ and choose
$M_2:=\left\lceil \sqrt{V_r(0)/u_{\min}}\right\rceil$. 
Since $V_r(j)\le V_r(0)$ and
$\|\bar r_i(j)-r_i(jM)\|\le V_r(j)$, we have $\|\Delta r\|\le V_r(0)$, hence
$\|\Delta r\|\le M_2^2 u_{\min}\le M_2^2 u_{i,\max}$, so the maneuver is admissible.

\emph{Step 3: feasibility and RI existence.}
Concatenating Step~1 and Step~2 yields a feasible $M=M_1+2M_2$ horizon that reaches the terminal
point $\xi=\bar r_i(j)$ with admissible inputs. This satisfies the terminal hull constraint
$\bar r_i(j)\in \C_i(j)$ and makes the terminal inequality trivial ($\|r_i(M)-\bar r_i(j)\|=0$).
Thus OCP~\eqref{eq:OCP} is feasible for all $(i,j)$.

If $\C_i(j)$ is not a singleton then $\bar r_i(j)\in\ri(\C_i(j))$ as in the single-integrator case,
hence $\bar r_i(j)\in F_i(j)\cap\ri(\C_i(j))$ and Assumption~\ref{ass:RIexist_ocp} holds.

\emph{Step 4: consensus.}
Now simply apply Lemma~\ref{lem:lex_RI} and Theorem~\ref{thm:consensus_ocp_lex}. 
\end{proof}

\section{Numerical Simulations}\label{sec:numerical-simulations}

In this section, we illustrate the performance of Algorithm~\ref{alg:dmproutine_lex} with lexicographic selection for agents with single- and double-integrator dynamics.
The simulations are performed using a distributed Python implementation based on a REQ/REP communication architecture~\cite{Hintjens.2013.zeromq,zeromq}, where a coordinator facilitates communication among the agents, each local controller with the corresponding agent.
All simulations were performed on a Windows 11 (64-bit) machine equipped with a 13th Gen Intel Core i5-1335U processor (10 cores, 1.30 GHz base frequency) and 32 GB RAM.
The implementation was carried out in Python 3.13.
Each agent solves its local OCP in parallel, while a coordinator aggregates the open-loop control sequences and propagates the agent dynamics over the finite-step horizon.

The objectives of the simulations are threefold:
\begin{enumerate}
    \item to verify the explicit finite-step bounds on $M$ derived in Corollary~\ref{cor:single_integrator_explicit} and Corollary~\ref{cor:double_integrator_explicit};
    \item to illustrate the asymptotic decay of the diameter function $V_r(j)$;
    \item to analyze the evolution of the interiority measure
    $
        \phi_{i,j} := \dist\!\bigl(x_i^\star(M;j), \partial \Cij \bigr),
    $
    and the activation of the lexicographic selection stage of Algorithm~\ref{alg:dmproutine_lex}.
\end{enumerate}

\subsection{Simulation setup}

We consider a network of $\ell=4$ agents in $\R^2$ with a time-invariant ring-form connected communication graph.
The initial conditions are chosen such that the initial convex hull is non-singleton and $V (0) > 0$.
For each outer step $j$, the agents solve their local OCP over a finite-step horizon $M$ as prescribed in Algorithm~\ref{alg:dmproutine_lex} with the lexicographic tolerance $\delta_\mathrm{lex} = 10^{-5}$.
We take the decay function $\alpha$ in the terminal inequality~\eqref{eq:OCP} linear as $\alpha_i (s)= 0.8 s$ and the control input bounds are chosen as $\|u_i\| \le 1$ for all $i=1,\dots,4$.

Picking the initial conditions for the position variables as
$x_1(0) = r_1(0) = (-4.0,  2.0)$, $x_2(0)=r_2(0)=( 3.5,  4.0)$, $x_3(0)=r_3(0)=(4.5,-3.5)$, $x_4(0)=r_4(0)=(-2.5, -4.0)$, as for the double-integrator the initial values for velocities $v_1(0) = (1.0,  0.0)$, $v_2(0)=(0.0, -1.0)$, $v_3(0)=(-1.0,  0.0)$, $v_4(0)=(0.0,  1.0)$, and the finite-step horizon $M$ is selected according to the explicit bounds:
\begin{itemize}
    \item Single-integrator (Corollary~\ref{cor:single_integrator_explicit}):
        $M = \left\lceil V(z(0)) \right\rceil = 11$;
    \item Double-integrator (Corollary~\ref{cor:double_integrator_explicit}):
        $M = \max\!\left\{2,\ \left\lceil 2v_{\max}\right\rceil
        + 2\left\lceil \sqrt{V_r(0)}\right\rceil \right\}
        = 12$.
\end{itemize}

\subsection{Single-integrator case}

We first consider the single-inetgrator dynamics $x_i(t+1) = x_i(t) + u_i(t)$ as in~\eqref{eq:si}.
Figure~\ref{fig:SI_states} shows the evolution of the position states over the inner steps. The outer-step instants are marked explicitly. The trajectories converge asymptotically to a common consensus value.
The corresponding control inputs are depicted in Figure~\ref{fig:SI_controls}. The control magnitudes remain bounded and decay as consensus is approached.
The evolution of the disagreement diameter function $V(z(j))$ is shown in Figure~\ref{fig:SI_Vr}.
As predicted by Corollary~\ref{cor:single_integrator_explicit}, $V(x(j))$ decreases monotonically and converges to zero.

\begin{figure}[t]
\centering
\includegraphics[width=0.9\linewidth]{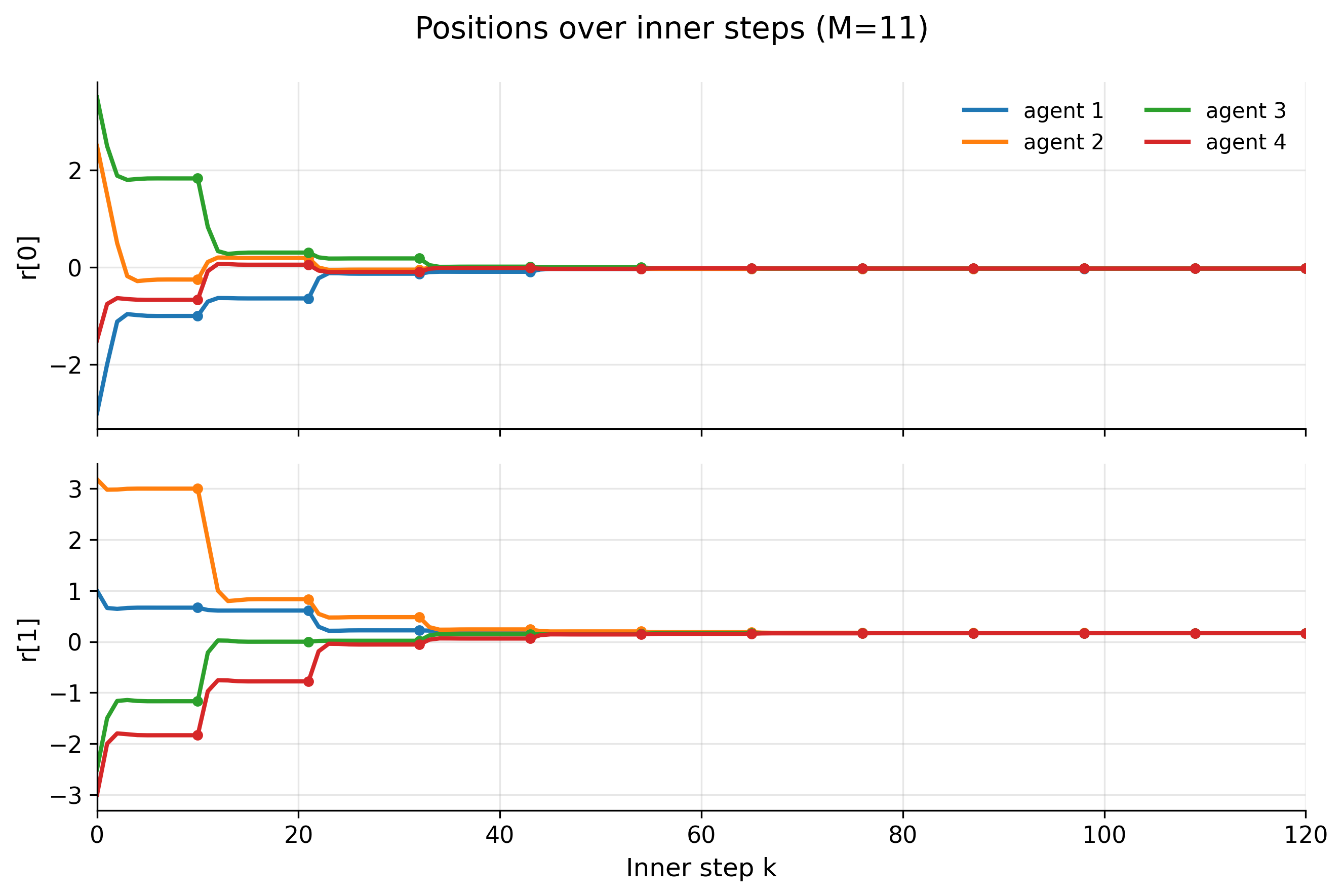}
\caption{State trajectories for (single-integrator).}
\label{fig:SI_states}
\end{figure}

\begin{figure}[t]
\centering
\includegraphics[width=0.9\linewidth]{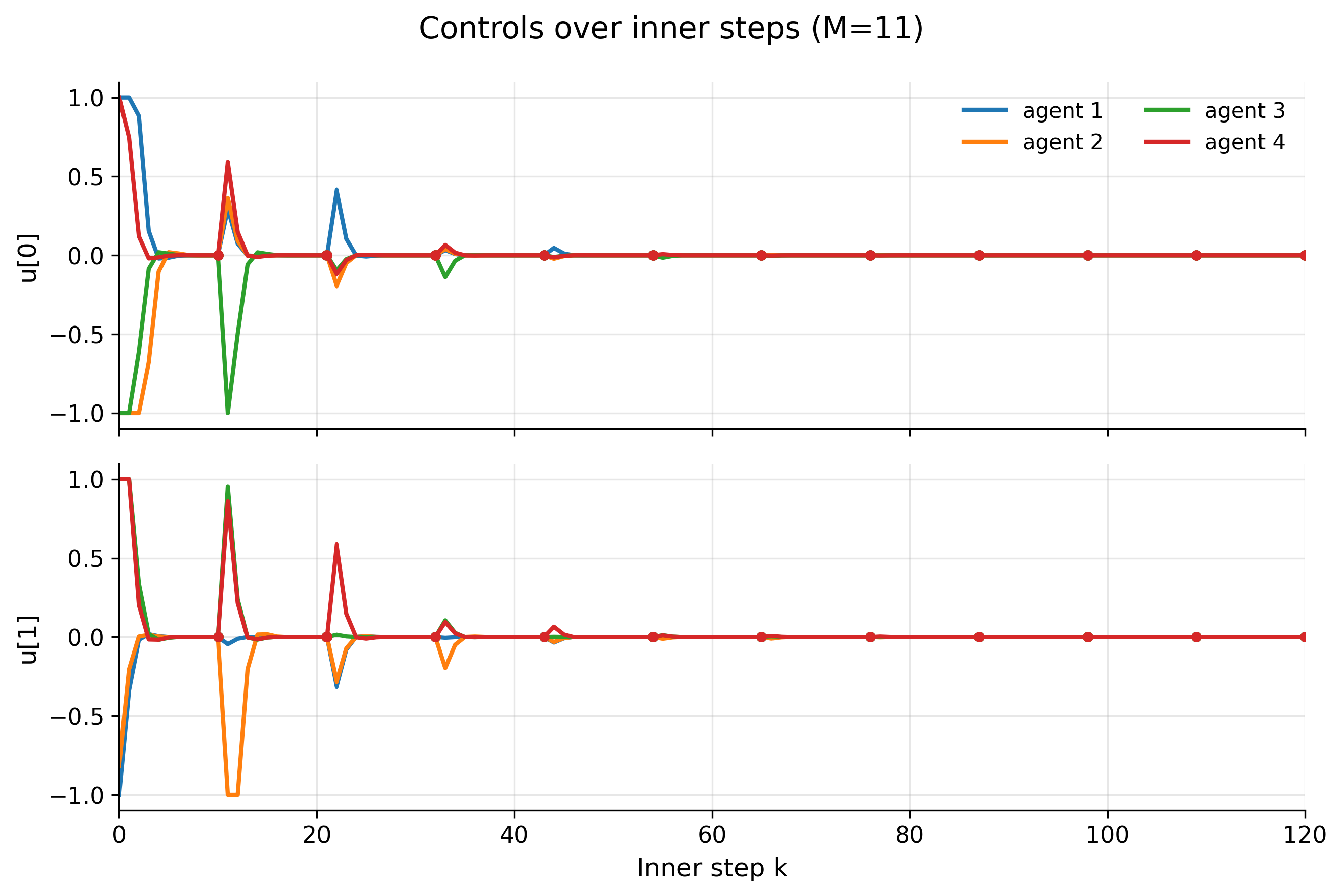}
\caption{Control inputs for (single-integrator).}
\label{fig:SI_controls}
\end{figure}

\begin{figure}[ht]
\centering
\includegraphics[width=0.9\linewidth]{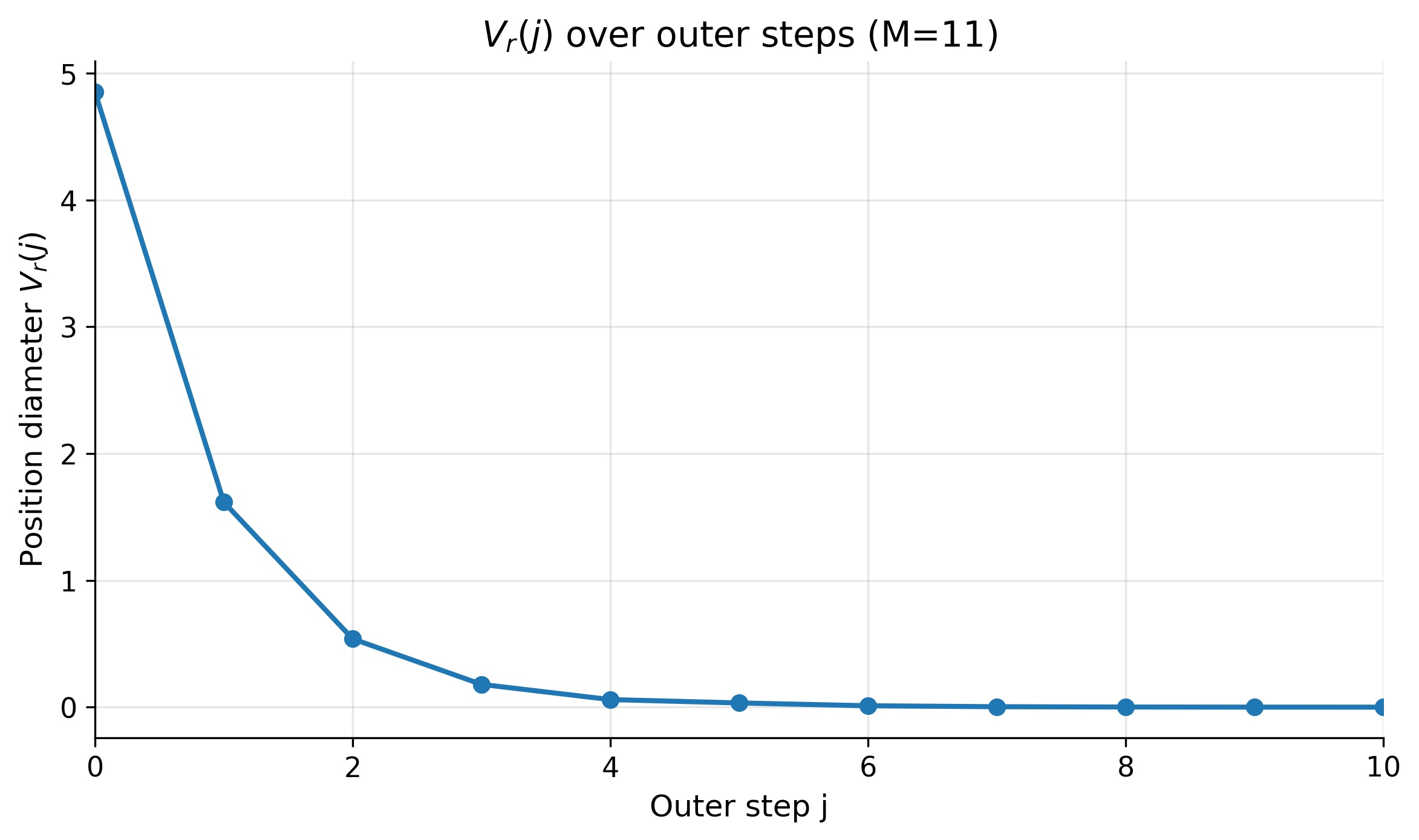}
\caption{Diameter function $V_r(j)$ over outer steps (single-integrator).}
\label{fig:SI_Vr}
\end{figure}

The interiority measure $\phi_{i,j}$ for each agent is reported in Figure~\ref{fig:SI_phi}. For all outer steps where $\mathcal C_i(j)$ is non-singleton, we observe $\phi_{i,j} > 0$,
indicating that the terminal point lies in the relative interior of $\mathcal C_i(j)$.
Consequently, the lexicographic stage of Algorithm~\ref{alg:dmproutine_lex} is not activated, as shown in Figure~\ref{fig:SI_lex}. The lexicographic flag remains zero throughout the simulation.

\begin{figure}[ht]
\centering
\includegraphics[width=0.9\linewidth]{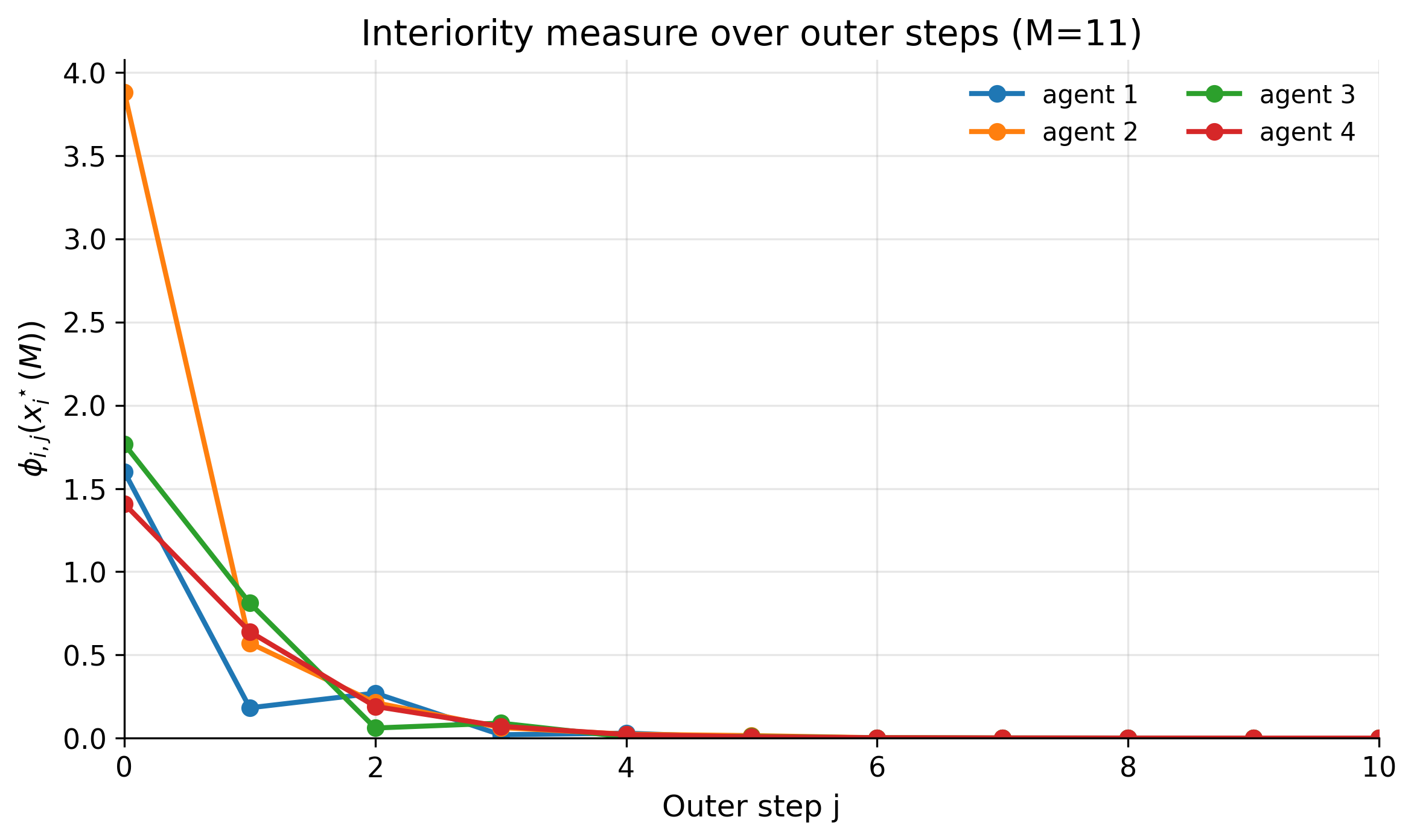}
\caption{Interiority measure $\phi_{i,j}$ (single-integrator).}
\label{fig:SI_phi}
\end{figure}

\begin{figure}[ht]
\centering
\includegraphics[width=0.9\linewidth]{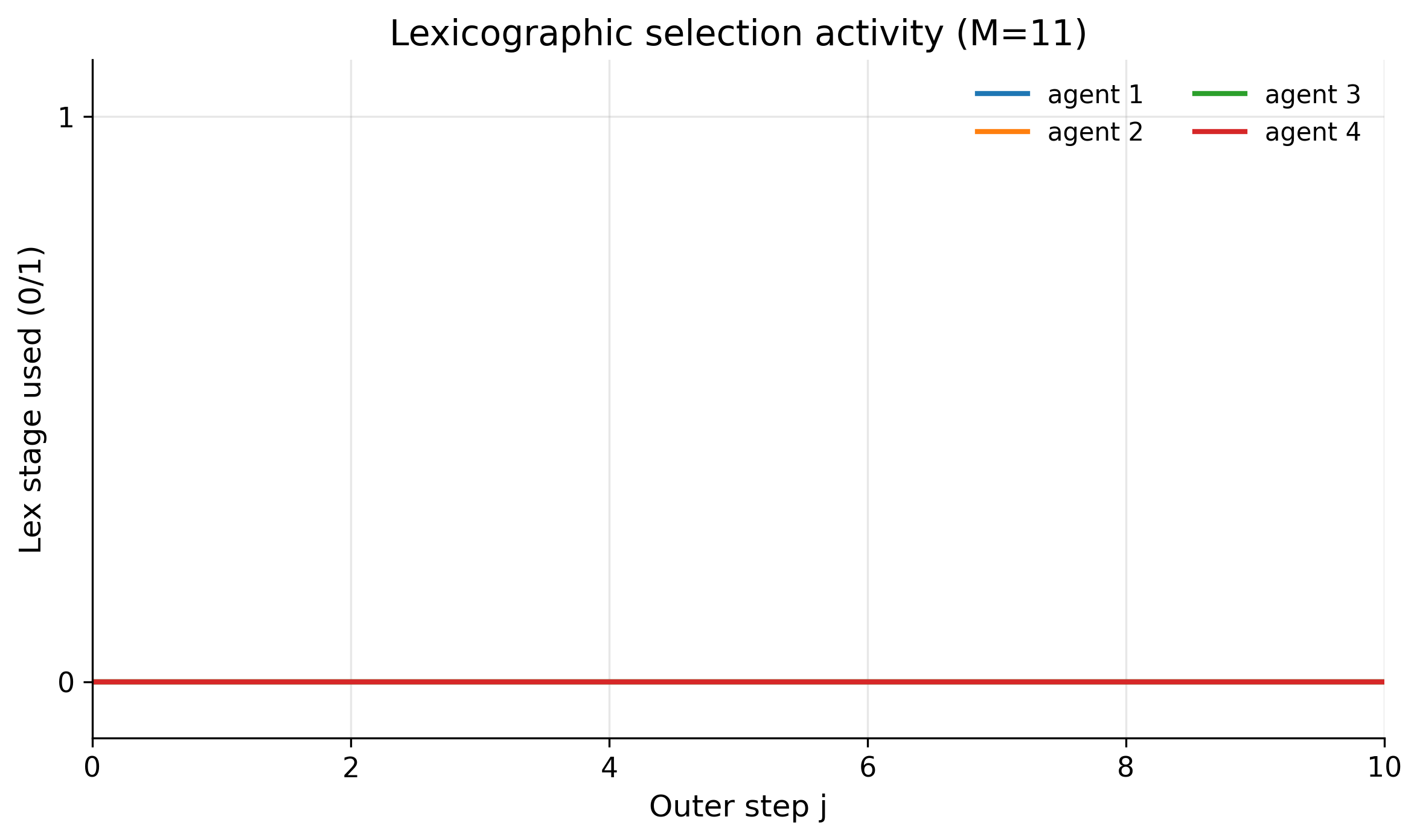}
\caption{Lexicographic activation flag (single-integrator).}
\label{fig:SI_lex}
\end{figure}
The primary cost $J_i^\star(j)$ decreases monotonically over outer steps, reflecting the contraction enforced by the terminal constraint (Figure~\ref{fig:SI_cost}).
\begin{figure}[t]
\centering
\includegraphics[width=0.9\linewidth]{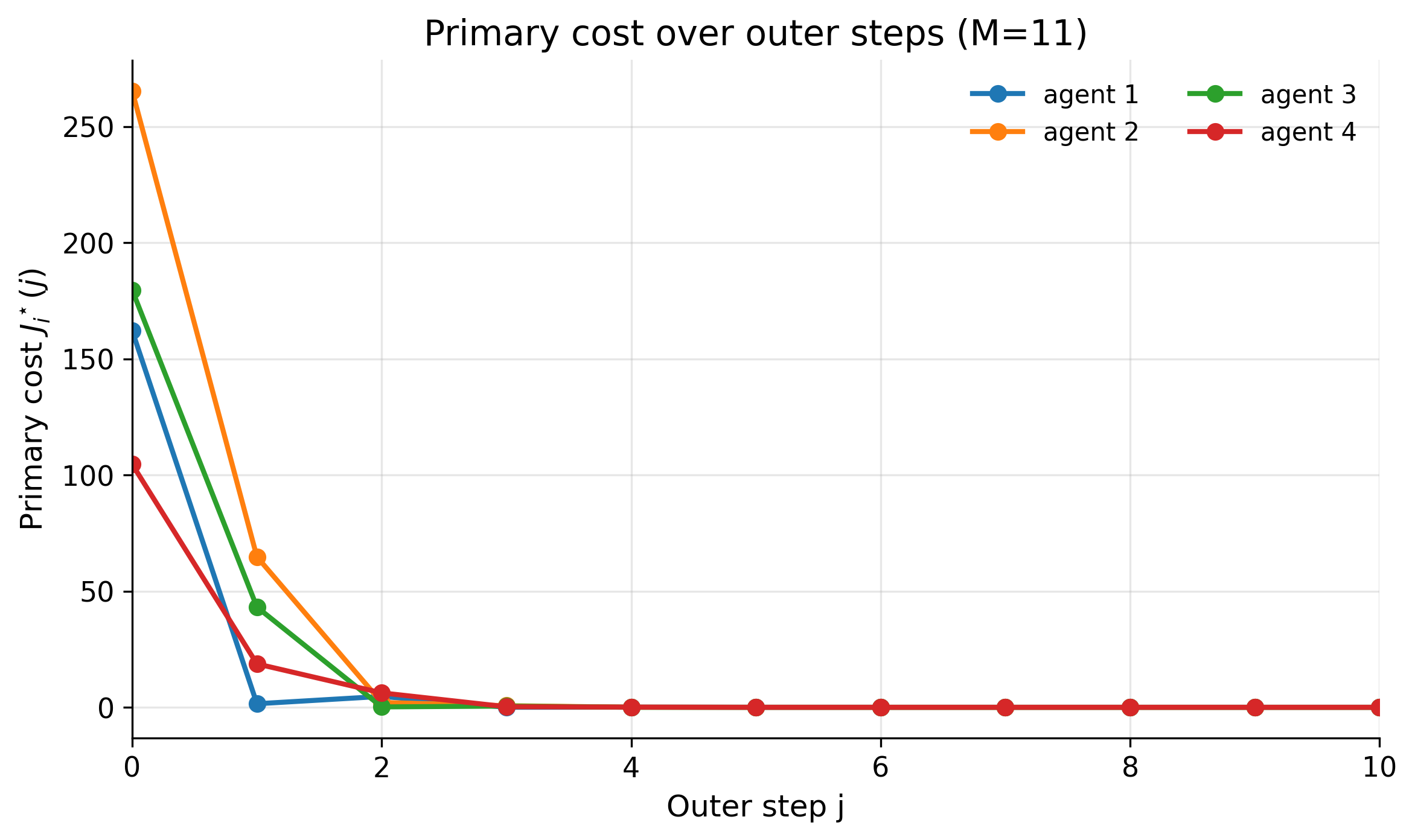}
\caption{Primary cost evolution (single-integrator).}
\label{fig:SI_cost}
\end{figure}
These results confirm the theoretical conclusions of Corollary~\ref{cor:single_integrator_explicit}: the explicit finite-step horizon $M=11$ suffices to guarantee asymptotic consensus without requiring activation of the lexicographic selection stage.

\subsection{Double-integrator case}

We next consider the dynamics
\[
    r_i(t+1)=r_i(t)+v_i(t), \qquad
    v_i(t+1)=v_i(t)+u_i(t).
\]

Figure~\ref{fig:DI_states} shows the position trajectories. Despite the second-order dynamics, asymptotic consensus in position is achieved.
\begin{figure}[t]
\centering
\includegraphics[width=0.9\linewidth]{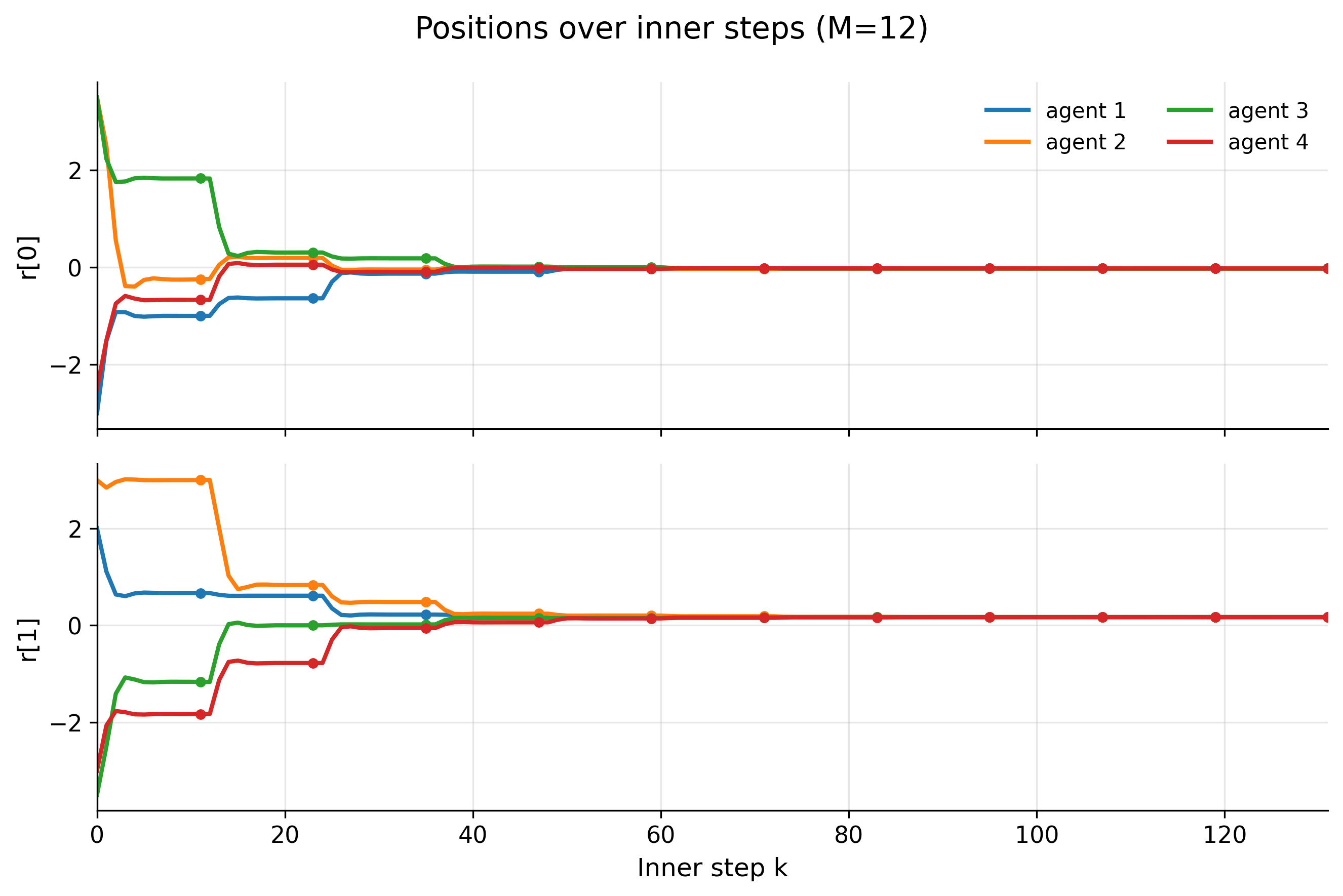}
\caption{Position trajectories for (double-integrator).}
\label{fig:DI_states}
\end{figure}
The velocity trajectories are displayed in Figure~\ref{fig:DI_vel}. As predicted by Corollary~\ref{cor:double_integrator_explicit}, velocities converge to zero.
The control inputs are shown in Figure~\ref{fig:DI_controls}.
The diameter function $V_r(j)$ is depicted in Figure~\ref{fig:DI_Vr}. Again, a monotonic decrease is observed, which confirms asymptotic consensus.
The interiority measure $\phi_{i,j}$ remains strictly positive whenever $\mathcal C_i(j)$ is non-singleton (Figure~\ref{fig:DI_phi}). Hence, the lexicographic stage remains inactive throughout the simulation (Figure~\ref{fig:DI_lex}).
The decay of the primary cost $J_i^\star(j)$ further confirms the contraction mechanism induced by the terminal constraint, see Figure~\ref{fig:DI_cost}.
\begin{figure}[t]
\centering
\includegraphics[width=0.9\linewidth]{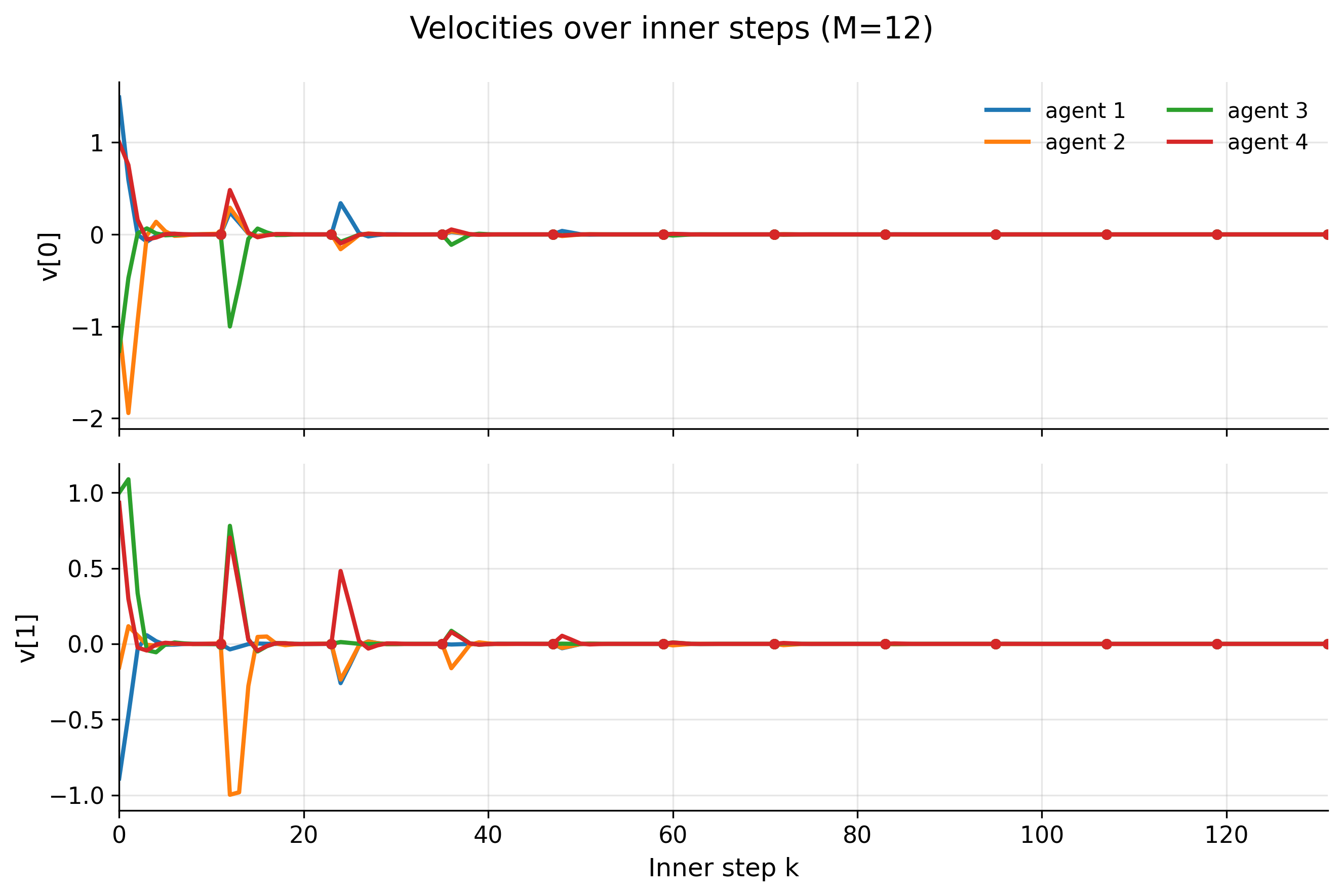}
\caption{Velocity trajectories (double-integrator).}
\label{fig:DI_vel}
\end{figure}
\begin{figure}[t]
\centering
\includegraphics[width=0.9\linewidth]{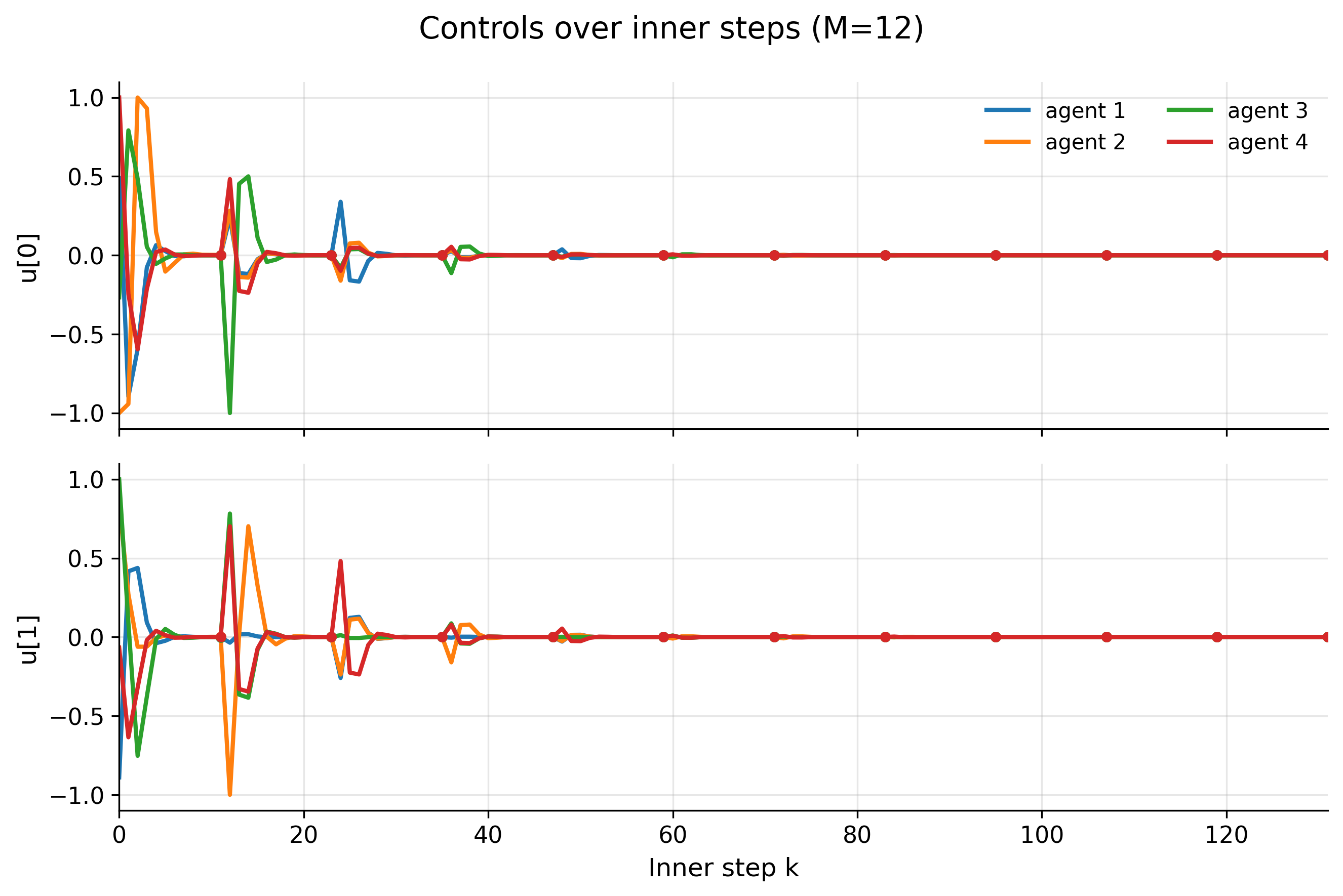}
\caption{Control inputs for (double-integrator).}
\label{fig:DI_controls}
\end{figure}

\begin{figure}[t]
\centering
\includegraphics[width=0.9\linewidth]{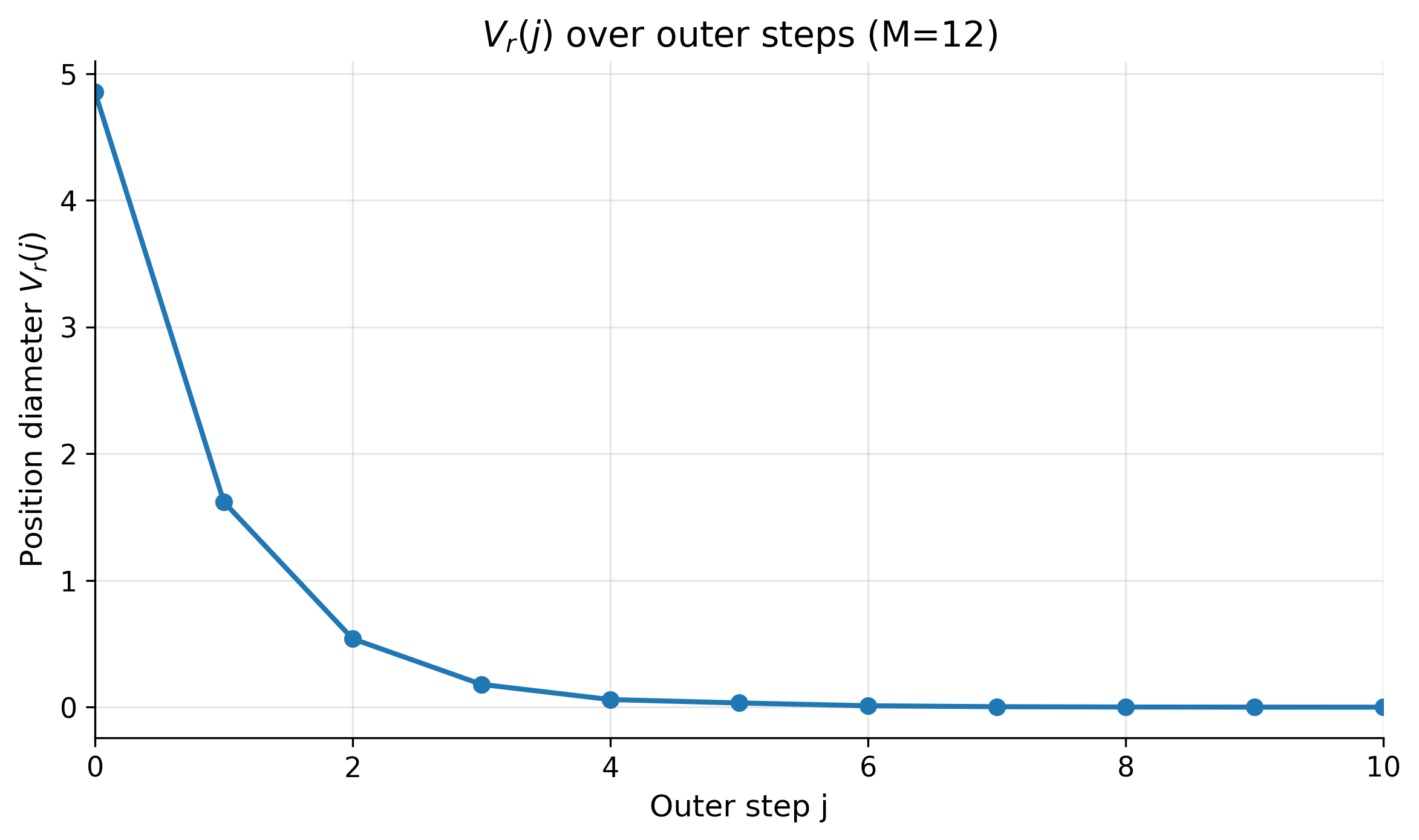}
\caption{Diameter function $V_r(j)$ (double-integrator).}
\label{fig:DI_Vr}
\end{figure}

\begin{figure}[t]
\centering
\includegraphics[width=0.9\linewidth]{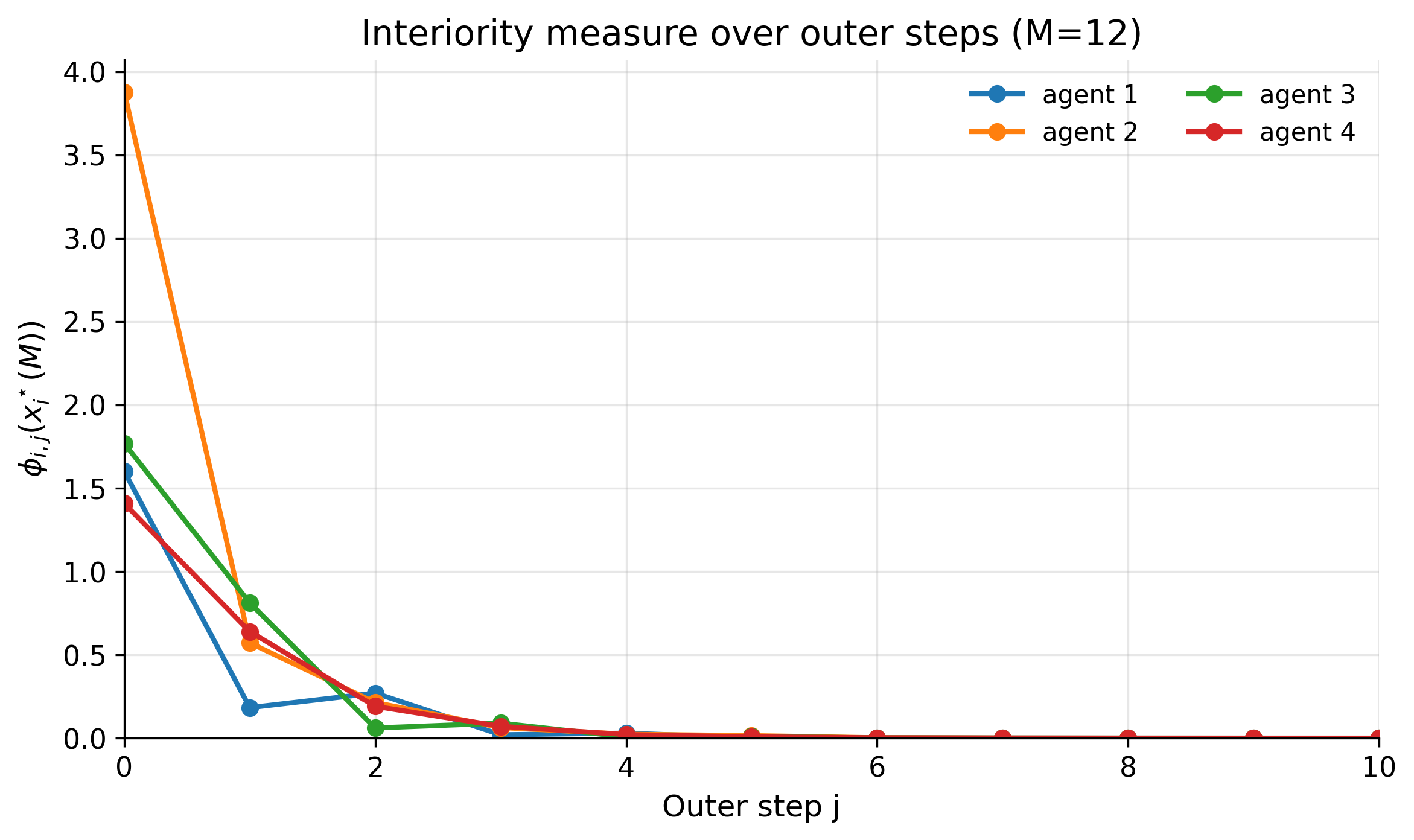}
\caption{Interiority measure $\phi_{i,j}$ (double-integrator).}
\label{fig:DI_phi}
\end{figure}

\begin{figure}[t]
\centering
\includegraphics[width=0.9\linewidth]{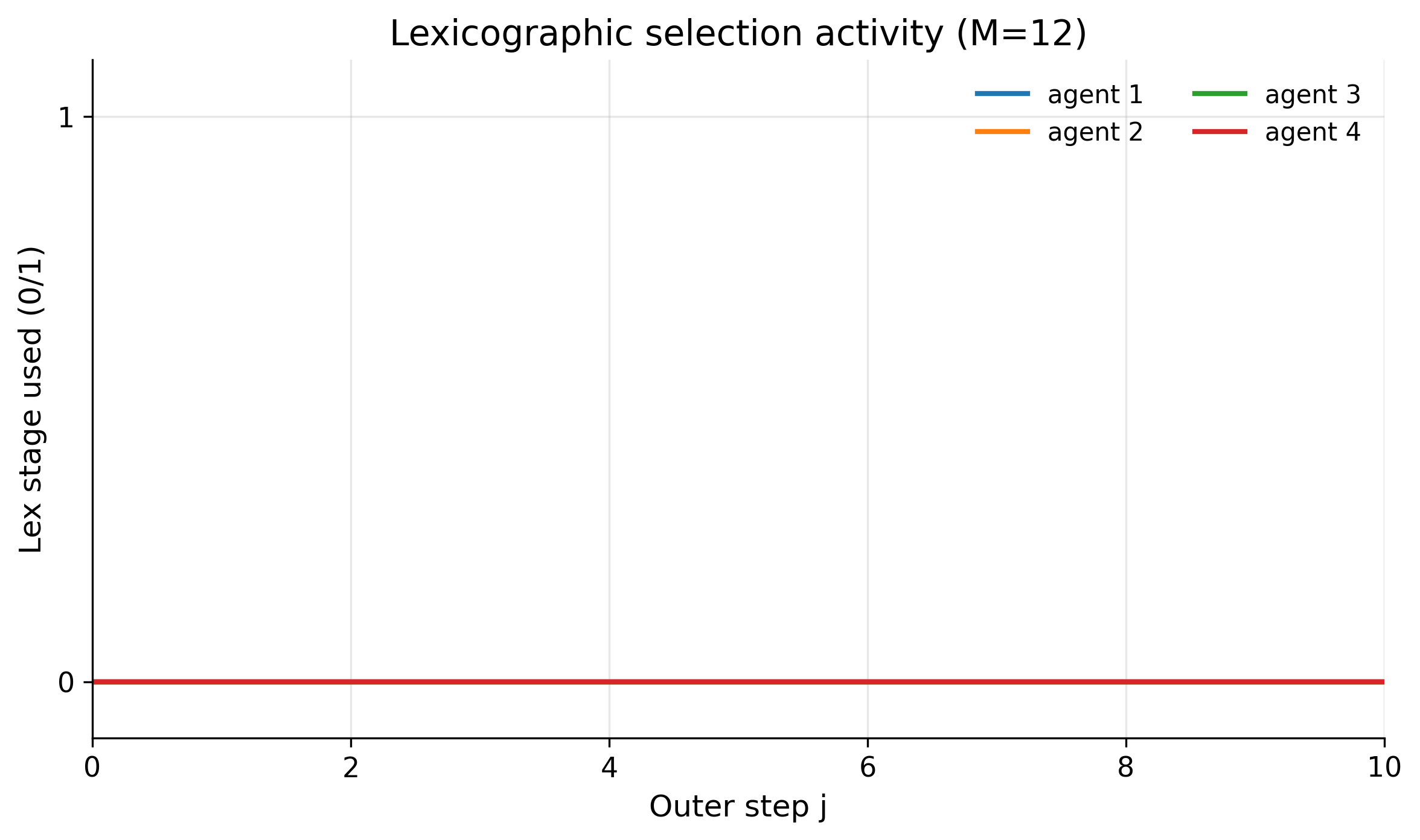}
\caption{Lexicographic activation flag (double-integrator).}
\label{fig:DI_lex}
\end{figure}

\begin{figure}[t]
\centering
\includegraphics[width=0.9\linewidth]{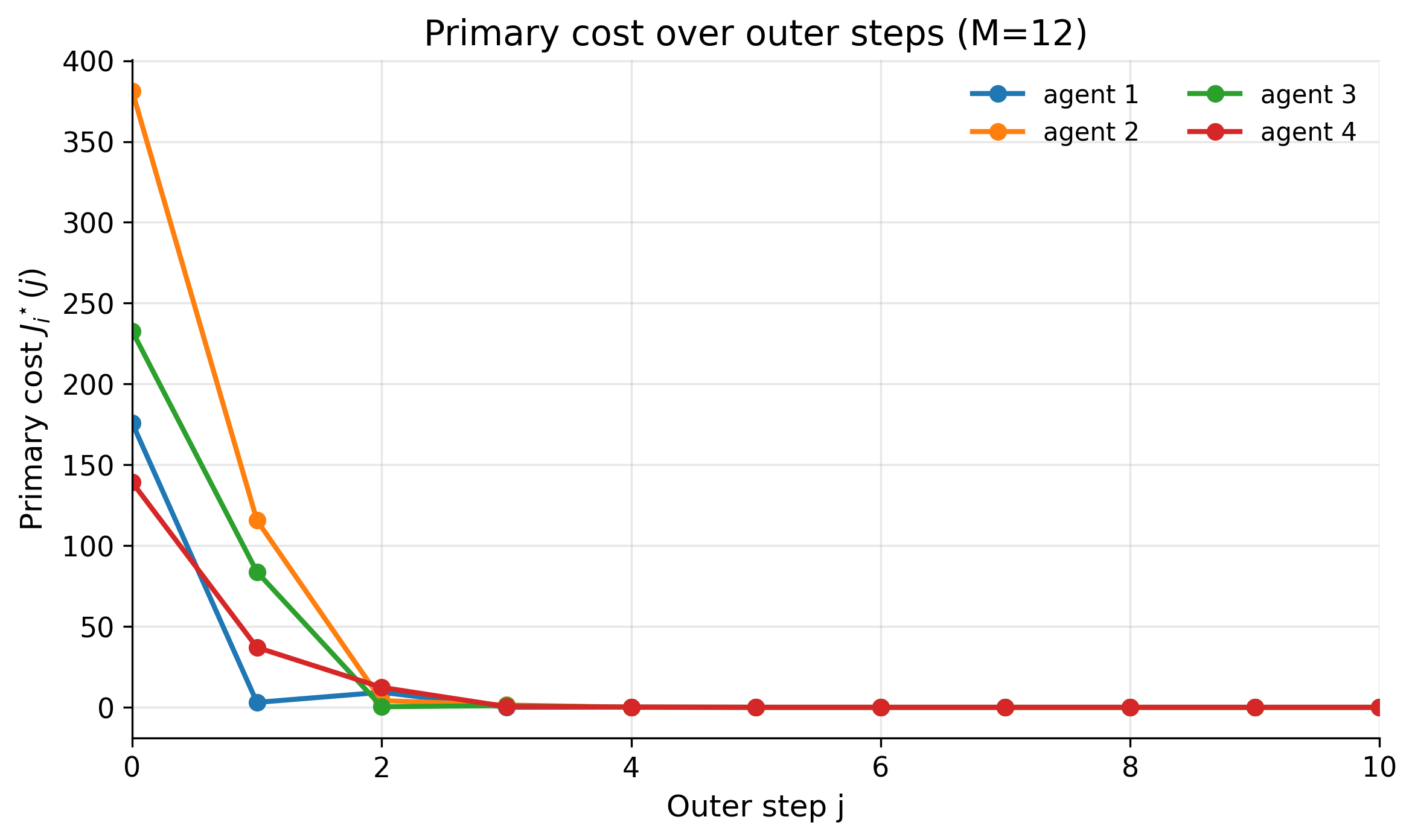}
\caption{Primary cost evolution (double-integrator).}
\label{fig:DI_cost}
\end{figure}

\subsection{Discussion}

The simulations demonstrate that:

\begin{itemize}
    \item The explicit finite-step bounds on $M$ derived in Corollaries~\ref{cor:single_integrator_explicit} and~\ref{cor:double_integrator_explicit} are sufficient in practice.
    \item The contraction of $V_r(j)$ is consistent with the theoretical finite-step Lyapunov function and strictness arguments underlying Theorem~\ref{thm:moreau_nes_suf}.
    \item The lexicographic selection stage (Step 7 of Algorithm~\ref{alg:dmproutine_lex} is not activated in these scenarios, since the primary optimization already yields terminal points in the relative interior of $\mathcal C_i(j)$.
\end{itemize}

Thus, the lexicographic stage serves as a safeguard mechanism ensuring strict interiority when necessary, while remaining inactive when the primary problem already enforces the desired strictness property.
Overall, the distributed finite-step MPC with lexicographic selection achieves asymptotic consensus in both single- and double-integrator settings, fully consistent with the theoretical development of Section~\ref{sec:linear-dynamics}.

\subsection{Computational complexity and implementation remarks}


For each agent $i$ and outer step $j$, the primary and (if activated) secondary problems are convex programs of fixed horizon length $M$. 
The problem size gives a standard and implementation-agnostic complexity proxy \emph{problem size}:
the number of scalar decision variables $n_{\mathrm{var}}$, the number of scalar equality constraints $n_{\mathrm{eq}}$, the number of scalar inequality constraints $n_{\mathrm{ineq}}$.

Since interior-point type solvers typically scale superlinearly in $n_{\mathrm{var}}$ (often heuristically like $\mathcal O(n_{\mathrm{var}}^3)$ for dense instances), we report, for each $(i,j)$, the tuple
$(n_{\mathrm{var}},n_{\mathrm{eq}},n_{\mathrm{ineq}})$
as a solver-independent complexity proxy, together with the measured wall-clock solve times
$T^{\mathrm{prim}}_{i}(j)$, $T^{\mathrm{lex}}_{i}(j)$,
$T^{\mathrm{tot}}_{i}(j)=T^{\mathrm{prim}}_{i}(j)+T^{\mathrm{lex}}_{i}(j)$,
where $T^{\mathrm{lex}}_{i}(j)=0$ if the lexicographic stage is not activated.
The tuple $(n_{\mathrm{var}},n_{\mathrm{eq}},n_{\mathrm{ineq}})$ quantifies the size of the local MPC problem and can be used to compare different horizons $M$, models (single- vs.\ double-integrator), and constraint sets. The solve-time traces $T^{\mathrm{prim}}_{i}(j)$ and $T^{\mathrm{lex}}_{i}(j)$ provide an empirical measure of the real-time feasibility of Algorithm~\ref{alg:dmproutine_lex} under the chosen solver and hardware.
Together, these indicators yield a practical complexity characterization of the distributed finite-step MPC implementation.
Table~\ref{tab:comp_complexity_si} reports the computational complexity for the single-integrator case.

\begin{table}[ht]
\centering
\caption{Per-agent computational complexity for single-integrators.}
\label{tab:comp_complexity_si}
\renewcommand{\arraystretch}{1.15}
\begin{tabular}{c|c|c|c|c}
\hline
Agent $i$ 
& \begin{tabular}{c}
$T_{\mathrm{total}}$ [ms]\\
{\footnotesize mean / max}
\end{tabular}
& \begin{tabular}{c}
$T_{\mathrm{primary}}$ [ms] \\
{\footnotesize mean / max}
\end{tabular}
& \begin{tabular}{c}
$T_{\mathrm{lex}}$ [ms] \\
{\footnotesize mean / max}
\end{tabular}
& $(n_{\mathrm{var}}, n_{\mathrm{eq}}, n_{\mathrm{ineq}})$ \\
\hline
1 & $50.2\,\,/\,\,77.0$ & $50.2\,\,/\,\,77.0$ & $0\,\,/\,\,0$ & $(25,3,96)$ \\
2 & $47.03\,\,/\,\,69.37$ & $47.03\,\,/\,\,69.37$ & $0\,\,/\,\,0$ & $(25,3,96)$ \\
3 & $46.90\,\,/\,\,69.09$ & $46.90\,\,/\,\,69.09$ & $0\,\,/\,\,0$ & $(25,3,96)$ \\
4 & $46.99 \,\,/\,\, 73.41$ & $46.99 \,\,/\,\, 73.41$ & $0\,\,/\,\,0$ & $(25,3,96)$ \\
\hline
\end{tabular}
\end{table}

Table~\ref{tab:comp_complexity_di} reports the computational complexity for the double-integrator case.

\begin{table}[ht]
\centering
\caption{Per-agent computational complexity for double-integrators.}
\label{tab:comp_complexity_di}
\renewcommand{\arraystretch}{1.15}
\begin{tabular}{c|c|c|c|c}
\hline
Agent $i$ 
& \begin{tabular}{c}
$T_{\mathrm{total}}$ [ms]\\
{\footnotesize mean / max}
\end{tabular}
& \begin{tabular}{c}
$T_{\mathrm{primary}}$ [ms] \\
{\footnotesize mean / max}
\end{tabular}
& \begin{tabular}{c}
$T_{\mathrm{lex}}$ [ms] \\
{\footnotesize mean / max}
\end{tabular}
& $(n_{\mathrm{var}}, n_{\mathrm{eq}}, n_{\mathrm{ineq}})$ \\
\hline
1 & $89.80 \,\,/\,\,111.3$ & $89.80 \,\,/\,\,111.3$ & $0\,\,/\,\,0$ & $(27,3,156)$ \\
2 & $89.09 \,\,/\,\, 108.41$ & $89.09 \,\,/\,\, 108.41$ & $0\,\,/\,\,0$ & $(27,3,156)$ \\
3 & $87.32 \,\,/\,\, 108.3$ & $87.32 \,\,/\,\, 108.3$ & $0\,\,/\,\,0$ & $(27,3,156)$ \\
4 & $88.46 \,\,/\,\, 114.7$ & $88.46 \,\,/\,\, 114.7$ & $0\,\,/\,\,0$ & $(27,3,156)$ \\
\hline
\end{tabular}
\end{table}

\section{Acknowledgment}
The author sincerely thanks Lars Grüne and Fabian Wirth for insightful and motivating discussions that initiated and inspired this research work.

\bibliographystyle{IEEEtran}
\bibliography{DistrubtedMutiStepMPC} 

\end{document}